\documentclass[11pt,leqno]{article}
\usepackage{graphicx}
\usepackage{wrapfig}
\usepackage{amsfonts, amsthm}
\usepackage{amsmath}
\usepackage[all]{xy}
\usepackage{blindtext}
\usepackage[inline]{enumitem}
\usepackage{xcolor}
\usepackage{url}
\begin{document}
\newtheorem{prop}{Proposition}[section]
\newtheorem{exa}{Example}[section]
\newtheorem{theo}{Theorem}[section]
\newtheorem{rem}{Remark}[section]
\newtheorem{lem}{Lemma}[section]
\newtheorem{defi}{Definition}[section]
\newtheorem{coro}{Corollary}[section]
\newtheorem{prob}{Problem}[section]
\newtheorem{nota}{Notation}[section]
\newtheorem{obs}{Observation}[section]
\def\proj{{\Bbb P}^2} 
\def\R{{\Bbb R}}
\def\Z{{\Bbb Z}}
\def\C{{\Bbb C}}
\def\Q{{\Bbb Q}}
\newcommand{\ffc}[1]{{{\cal F}}(#1)}
\newcommand{\wc}[1]{{{\cal M}}(#1)}
\newcommand{\ps}[1]{{{\Bbb P}}^{#1}}
\def\F{{\cal F}}
\def\P{{\cal P}}
\def\Pt{\C[t]}
\def\RR{{\cal R}} 
\def\li{{\Bbb L}}
\def\L{{\cal L}}
\newcommand{\cfc}[1]{\{\delta_t\}_{t\in #1}}
\begin{center}
{\LARGE\bf Center conditions: pull-back of differential equations
\\}
\vspace{.25in} {\large {\sc Yadollah Zare}}

\end{center}
\begin{abstract}
The space of polynomial differential equations of a fixed degree with a center
singularity has many irreducible components. We prove that pull-back differential equations form an irreducible component of such a space. The method used in this article is inspired by Ilyashenko and Movasati's method. The main concepts are the Picard-Lefschetz theory of a polynomial in two variables with complex coefficients, the Dynkin diagram of the polynomial and the iterated integral.
\end{abstract}
\setcounter{section}{-1}
\section{Introduction}
Let $\mathbb{C}[x,y]_{\leq d}$ be the set of polynomials in the two variables $x,y$, and coefficients in $\mathbb{C}$ of degree less than or equal to $d\in\mathbb{N}_0$. The space of algebraic foliations
\[
\mathcal{F}=\mathcal{F}(\omega)\ \ ,\ \ \  \omega\in \Omega_d^1 ,
\]   
where 
\[
\Omega_d^1:=\{P(x,y)dy-Q(x,y)dx|\ \ P,Q\in \mathbb{C}[x,y]_{\leq d}\},
\]
is the projectivization of the vector space $\Omega_d^1$, and it is denoted by $\mathcal{F}(d)$. The maximum degree of the polynomials $P$ and $Q$ is known as the (affine) degree of $\mathcal{F}$. The set of singularities of the foliation $\mathcal{F}$ is $ V(P)\cap V(Q)$. If $(P_xQ_y-P_yQ_x)(p)\neq 0$, for an isolated singularity $p$ of $\mathcal{F}$, then $p$ is called \emph{reduced} singularity. If there is a holomorphic coordinate system $(\tilde{x},\tilde{y})$ in a neighborhood of a reduced singularity $p$ with $\tilde{x}(p)=0$, $\tilde{y}(p)=0$ such that in this coordinate system
$$
\omega\wedge d(\tilde{x}^2+\tilde{y}^2)=0,
$$ 
then the point $p$ is called a \emph{center singularity}. 
The closure of the set of algebraic foliations of fixed degree $d$ with at least one center in $\mathcal{F}(d)$, which is denoted by $\mathcal{M}(d)$, is an algebraic subset of $\mathcal{F}(d)$ (see for instance, \cite{14} and \cite{12}). Identifying irreducible components of $\mathcal{M}(d)$ is the center problem in the context of polynomial differential equations in the real plane.
The complete classification of irreducible components of $\mathcal{M}(2)$ is done by 
H. Dulac in \cite{5} (see also \cite[p.601]{3}). This classification has applications to Hilbert's 16th problem. Ilyashenko in \cite{10}, by computing tangent space at some smooth points of the space of Hamiltonian foliations $\mathcal{F}(df)$,  
$
f\in\mathbb{C}[x,y]_{\leq d+1}
$, proved the following:
\begin{theo}\label{Ilya}
The space of Hamiltonian foliations of degree $d$ forms an irreducible component of $\mathcal{M}(d)$.
\end{theo}
H. Movasati in \cite{16}, by computing the tangent cone of $\mathcal{M}(d)$ at a special point proved the following: 
\begin{theo}
The space of $\mathcal{L}(d_1,d_2,\cdots,d_s)$ of logarithmic foliations 
\begin{eqnarray*}
\mathcal{F}(\omega)&\omega=f_1\cdots f_s\sum_{i=1}^s\lambda_i\frac{df_i}{f_i},
\end{eqnarray*} 
$$
f_i\in\mathbb{C}[x,y]_{\leq d_i},\ \ \  \lambda_i\in\mathbb{C}, \ \  \ i=1,2,\cdots ,s,  \ \ \  d=\sum_{i=1}^s d_i -1
$$
is an irreducible component of $\mathcal{M}(d)$.
\end{theo}
Let $\mathcal{P}(a,n)$ be the set of foliation
\begin{equation}
\mathcal{F}(F^*(\omega)) \ \ \   where \ \ \  \omega\in\Omega_a^1 \ , 
\end{equation}
\begin{equation*}
 F:\mathbb{C}^2\rightarrow\mathbb{C}^2 \ \ is\ defined \ by\ (x,y)\rightarrow(R,S) \ \ and \  R,S\in\mathbb{C}[x,y]_{\leq n}   \ , \ n \geq 2.
\end{equation*}
For a generic morphism $F$ and foliation $\mathcal{F}$, there exists a leaf of $\mathcal{F}$ such that it has an intersection with $F(D)$ at some points with multiplicity 2, where $D$ is the curve $V(R_xS_y-R_yS_x)$. Therefore, $F^*(\mathcal{F})$ has a center singularity.
In this paper we are going to show the following;
\begin{theo}\label{1:28}
The space $\mathcal{P}(a,n)$ of pull-back differential equations 
$$
\mathcal{F}(\omega) \ , \ \omega=P(R,S)dS-Q(R,S)dR ,
$$
where
$$
R,S\in\mathbb{C}[x,y]_{\leq n} \ , \ P,Q \in \mathbb{C}[x,y]_{\leq a}
\ , \ d=an+n-1 \ , \ n \geq 2,
$$
forms an irreducible component of $\mathcal{M}(d)$.
\end{theo}

This paper is inspired by Ilyashenko's  paper \cite{10} and H. Movasati's paper \cite{16}. A sketch of our proof is the following: \\ 
Consider a generic $F$ and a generic polynomial $f\in\mathbb{C}[x,y]$ of degree $a+1$. It is clear that the point $\mathcal{F}:=\mathcal{F}(d(f\circ F))$ is in the intersection of $\mathcal{H}(an+n-1)$ and $\mathcal{P}(a,n)$ of the algebraic set $\mathcal{M}(an+n-1)$.
It is needed to show that the tangent
cone of $\mathcal{M}(an+n-1)$ at the point $\mathcal{F}$ is equal to $T_{\mathcal{F}}{\mathcal{H}}(an+n-1)\cup T_{\mathcal{F}}{\mathcal{P}}(a,n)$, in order to prove Theorem \ref{1:28}. 
This paper is organized as follows:

In \S 1, by taking a deformation 
$d(f\circ F)+\epsilon^k\omega_k+\epsilon^{k+1}\omega_{k+1}+\cdots+\epsilon^{2k}\omega_{2k}+h.o.t,$ of $d(f\circ F)$, where $\omega_k\neq0$, and using Petrov module concept, we show that there is a polynomial 1-form $\alpha\in\Omega^1$ of degree $a$ and a polynomial $K\in\mathbb{C}[x,y]$ such that $\omega_k$ is of the form $F^*(\alpha)+dK$.  
 
In \S 2, we calculate the explicit form of $dK$, by using the iterated integral and Melnikov function $M_{2k}$. This gives us the proof of Theorem \ref{1:28}.

In \S 3, we see some applications of theorem \ref{1:28}. We find a maximum lower bound for the cyclicity of a tangency vanishing cycle in a deformation $\mathcal{F}$ inside $\mathcal{F}(d)$, which is dependent on a factorization of $d$ into a product of two natural numbers.

In  \S 4, we study the action of the monodromy group on a tangency vanishing
cycle in a regular fiber $f\circ F$. \\

\section{Pull-back of differential equations }
Inspired by H. Movasati's method (see \cite{16}), we will calculate the tangent cone of $\mathcal{M}(n(a+1)-1)$ at a point in the intersection of Hamiltonian and pull-back algebraic differential equations. Similar to \cite{16} and \cite{10} our methods are based on Picard-Lefschetz theory for a foliations with a first integral.

Let $\mathcal{F}:=\mathcal{F}(\omega)\in\mathcal{F}(a)$ be a foliation of degree $a$, and $F=(R,S):\mathbb{C}^2\rightarrow\mathbb{C}^2$ be a morphism, where  $R,S\in\mathbb{C}[x,y]_{\leq n}$ and  $n\geq 2$. If a point $q$ is a tangent point of $F(D)$ and a leaf of the foliation $\mathcal{F}$, then a point in $F^{-1}(q)$ becomes a center singularity and it is called a \emph{tangency} critical point of the foliation $F^*(\mathcal{F})$.
\begin{theo}\label{Main theorem}
Consider the deformation $\mathcal{F}_{\epsilon}:=\mathcal{F}(\omega_\epsilon)$
$$
\omega_{\epsilon}=F^{*}(\omega)+\epsilon \omega_1+\cdots \ , \ deg(\mathcal{F}_{\epsilon})\leq d \ , \ d=an+n-1, \ n\geq 2
$$ 	
of the foliation $F^{*}(\mathcal{F}(\omega))$. Let $p$ be one of the tangency critical points of foliation $F^{*}(\mathcal{F}(\omega))$. For a generic \footnote{By generic we mean always a non-empty Zariski open subset of the ambient space. } choice of $\omega$ and $F$, if the deformed foliation $\mathcal{F}_{\epsilon}$ for all small $\epsilon $ has a center singularity near $p$, then $\mathcal{F}_{\epsilon}$
 is also a pull-back foliation. More precisely, there is a foliation $\tilde{\mathcal{F}}_{\epsilon}\in\mathcal{F}(a)$ and a polynomial map $F_{\epsilon}=(R_{\epsilon},S_{\epsilon}):\mathbb{C}^2\rightarrow\mathbb{C}^2$ 
 $R_{\epsilon} , S_{\epsilon}\in\mathbb{C}[x,y]_{\leq n}$ such that 
 $$
 F_{\epsilon}^{*}\tilde{\mathcal{F}}_{\epsilon}=\mathcal{F}_{\epsilon} \ , \ F_{0}=F \ \ , \ \ \mathcal{F}_0=\mathcal{F}.
 $$    
\end{theo}
Note that Theorem \ref{Main theorem} is equivalent to Theorem \ref{1:28}. This theorem is proved in the section \ref{18:51:03}.
\subsection{Tangent space }\label{14:02} 
The set $\mathcal{P}(a,n)$ is an irreducible algebraic subset of $\mathcal{M}(an+n-1)$ (by taking the coefficient of the  polynomials as coordinates of the  map from the space of polynomials with degree $an+n-1$ to the projective space). We are going to show that $\mathcal{P}(a,n)$ is also a component of $\mathcal{M}(an+n-1)$. Let us take a point $\mathcal{F}$ of $\mathcal{P}(a,n)$, then make a deformation $\mathcal{F}_{\epsilon}\in\mathcal{P}(a,n)$ and calculate the tangent vector space of $\mathcal{P}(a,n)$ at $\mathcal{F}$:
\begin{equation}\label{13:59:03}
 \begin{split}
 F_{\epsilon}^{*}(\omega_{\epsilon})&=(F+\epsilon F_1)^{*}(\omega+\epsilon  \alpha_1)+O(\epsilon^2)\\
 &=F^{*}(\omega)+\epsilon W + O(\epsilon^2),
 \end{split}
\end{equation}
where
\begin{gather*}
W=P(R,S)dS_1-Q(R,S)dR_1 +\\
R_1(\frac{\partial P}{\partial x}(R,S)dS-\frac{\partial Q}{\partial x}(R,S)dR)+S_1(\frac{\partial P}{\partial y}(R,S)dS-\frac{\partial Q}{\partial y}(R,S)dR)+F^{*}(\alpha_1).
\end{gather*}
For a smooth point $\mathcal{F}$ of $\mathcal{P}(a,n)$,
the tangent space of $\mathcal{P}(a,n)$ at $\mathcal{F}$ is just the set of all vectors $W$, which is contained 
in the tangent space of $\mathcal{M}(d)$ at $\mathcal{F}$, and in order to prove our main theorem it is enough to prove that the equality happens.  
Now, we are going to compute the tangent cone of $\mathcal{M}(an+n-1)$ at a point in the intersection of Hamiltonian
component and the set $\mathcal{P}(a,n)$.
\subsection{A foliation in the intersection of two algebraic sets}\label{15:46} 
A function $f:\mathbb{C}^2\rightarrow\mathbb{C}$ is called Morse if f has only non-degenerate critical points with distinct critical values.  
Let $F:=\mathbb{C}^2\rightarrow\mathbb{C}^2$ be defined by 
\begin{eqnarray}\label{12:35}
(x,y)\mapsto(R(x),S(y)):=\left(\prod_{i=1}^n(x-t_i)\ , \ \prod_{j=1}^n(y-t_j')\right),
\end{eqnarray}
where  $t_i \ , \ t_j'\in\mathbb{R}^{\geq 0}$  , $R$ and $S$ are Morse functions. Let $g $ and $h$ be two polynomials of degree $a+1$ defined by 
\begin{eqnarray}\label{13:08}
g(x):=\prod_{i=1}^{a+1}(x-s_i), &and& h(y):=\prod_{j=1}^{a+1}(y-s_j'),
\end{eqnarray}
and assume that they meet the following conditions:
\begin{enumerate}
\item For every $i$ and $j$ the roots $s_i \ , s_j' $ are positive real numbers.
\item Both equations $R(x)=s_i$ and $S(y)=s_j'$ have $n$ real roots.
\item The functions $g,h $ are \emph{Morse} polynomials and $ g\circ R$ and $h\circ S$ have only non-degenerate critical points with simple roots.
\item If $p$ is a critical point of $R$ (resp. $S$) and $q\in R^{-1}(q_1)$ (resp. $q\in S^{-1}(q_1)$) where $q_1$ is a critical point of $g$ (resp. $h$), then $|g\circ R(p)|>|g\circ R(q)|$ (resp. $|h\circ S(p)|>|h\circ S(q)|$)\footnote{This condition is needed for computation of the intersection form of vanishing cycles in Theorem \ref{14:06}}. This means that the critical values of $g$ are more close to zero than the critical values of $R$. In fact, by moving the roots of $g$ and $h$ on the real line this is the assumable definition.  
\end{enumerate} 
 Let  $f\in\mathbb{C}[x,y]_{\leq a+1}$ be defined by 
\begin{eqnarray}\label{12:51} 
 f(x,y):=g(x)+h(y). 
\end{eqnarray} We can assume that the intersection of the set of the critical values of $g\circ R$ and $h\circ S$ is empty. 
Note that all the singularities of the foliation $\mathcal{F}_0$ are real center type, indeed, they divide into three groups:
\begin{enumerate}
\item Pull-back of centers of $\mathcal{F}(df)$ under $F$,
\item Tangency critical points which are the preimage of the tangent points of leaves of the foliation $\mathcal{F}(df)$ with $F(V(R_x\cdot S_y))$, 
\item The points in $V(R_x)\cap V(S_y)$. 
\end{enumerate}   
Let $X(a,n)$ be the irreducible component of $\mathcal{M}(an+n-1)$ containing $\mathcal{P}(a,n)$.\\
Consider the deformation 
\begin{eqnarray}\label{17:30}
 \omega_{\epsilon}:d(f\circ F)+\omega_k\epsilon^{k}+\omega_{k+1}\epsilon^{k+1}+\cdots \ , \ deg(\omega_i)\leq d,
\end{eqnarray}
of $d(f\circ F)$.
Assume that $\mathcal{F}_{\epsilon}:=\mathcal{F}(\omega_\epsilon)$ belongs to $X(a,n)$.
This implies that $\mathcal{F}_{\epsilon}$ always has a center singularity near a fixed tangency center $p$ of $\mathcal{F}_0$. The set of all differential forms $\omega_k$
 is the tangent cone of $M(an+n-1)$ at $\mathcal{F}$. Note that taking $k=1$ is not sufficient for calculating the tangent cone. 
 
Let $\delta_t$ be a continuous family of vanishing cycles around a tangency critical point $p$ and $\Sigma$ be a transverse section to 
$\mathcal{F}$ at some point of $\delta_t$. We are able to write the Taylor expansion of the deformed holonomy $h_{\epsilon}(t)$
$$
h_{\epsilon}(t)-t=M_1(t)\epsilon+M_2\epsilon^2+\cdots+M_i(t)\epsilon^i+\cdots
.$$
Here $M_i(t)$ is the i-th Melnikov function of the deformation. Since $\omega_i=0$ for $0\leq i\leq k-1$, then 
$$
M_1=M_2=\dots=M_{k-1}=0.
$$       
If $\Sigma $ is parametrized by the image of $f$, i.e., $t=f(z)$ , $z\in\Sigma$ then 
\begin{equation} \label{15/09}
M_k(t)=-\int_{\delta_t}\omega_k \ .
\end{equation}
See for instance \cite{6}. Since $\omega_\epsilon$ is a deformation of $d(f\circ F)$ then it has center near tangency critical point $p$. Therefore $h_\epsilon$ is identity and $M_k(t)=0$. In fact, if $\lambda\in\pi_1(\mathbb{C}\setminus C,b)$ is a monodromy then by using analytic continuation
we can have 
$$\int_{\lambda(\delta_t)}\omega_k=0.$$
Here, $C$ is the set of critical values of the function $f\circ F$.
\begin{theo}\label{1:29:1}
The morphism $F_*:H_1((f\circ F)^{-1}(b),\mathbb{Z}) \rightarrow H_1(f^{-1}(b),\mathbb{Z})$ is surjective and $ker(F_*)$ is a group generated by the action monodromy group $\pi_1(\mathbb{C}\setminus C,b)$ on a vanishing cycle around a tangency point. 
\end{theo}
We will prove this theorem at the end of  \S 4, see Theorem \ref{18:00}. 
\subsection{Brieskorn/Petrov modules}
Consider the Brieskorn Modules/Petrov module 
\begin{eqnarray*}
H_f:=\frac{\Omega_{\mathbb{C}^2}^1}{df\wedge\Omega_{\mathbb{C}^2}^0+d\Omega_{\mathbb{C}^2}^0} & and &
H_{f\circ F}:=\frac{\Omega_{\mathbb{C}^2}^1}{d(f\circ F)\wedge\Omega_{\mathbb{C}^2}^0+d\Omega_{\mathbb{C}^2}^0}
\end{eqnarray*} 
where $H_f$ and $H_{f\circ F}$ are $\mathbb{C}[s]$-module and $\mathbb{C}[s']$-module 
respectively (here $s=f$ and $s'=f\circ F$ , and also $\Omega_{\mathbb{C}^2}^i$, $i=0,1$ are the set of polynomial differential forms in $\mathbb{C}^2$). 
\begin{defi}
A polynomial $l\in\mathbb{C}[x,y]$ of degree $d$ with homogeneous leading part $l_d$ is called \emph{transversal} to infinity, if $l_d$ factors out as the product of d pairwise different linear forms.   
\end{defi}   

Consider the Milnor module 
$$
V_{l_d}=\frac{\mathbb{C}[x,y]}{<(l_d)_x,(l_d)_y>},
$$ 
with the basis $\{g_\beta|\beta\in I_l\}$ where $g_\beta=x^{\beta_1}y^{\beta_2}$ and $I_l=\{\beta=(\beta_1,\beta_2)|for\ some\ 0\leq \beta_1+\beta_2\leq 2d-4\}$ of cardinality $(d-1)^2$.
We define 
\begin{equation*}
\begin{split}
A_{\beta}&:=\frac{\beta_1+1}{d}+\frac{\beta_2+1}{d},\\
\eta&:=xdy-ydx,\\
\eta_{\beta}&:=g_{\beta}\eta,
\end{split}
\end{equation*}
where $\beta=(\beta_1,\beta_2)$. 
\begin{theo}\label{15:39:1}
Let $l(x,y)\in\mathbb{C}[x,y]$ be a polynomial transversal to infinity of degree $d$. The $\mathbb{C}[l]$-module $H_l$ is free and $\eta_{\beta}$, where $\beta\in I_l$, forms a basis of $H_l$. Furthermore, every $\omega\in \Omega_{\mathbb{C}^2}^1$ can be written 
\begin{eqnarray}
\omega=\sum_{\beta} {h_{\beta}(l) \eta_{\beta}} + dl\wedge \zeta_1 +d\zeta_2 ,
\end{eqnarray}
\begin{eqnarray*}
h_{\beta}\in\mathbb{C}[l]\ , \ \zeta_1 , \zeta_2\in\mathbb{C}[x,y]\ \ , \ \ deg(h_{\beta})\leq \frac{deg(\omega)}{d}-A_{\beta}.
\end{eqnarray*}
\end{theo}
{See e.g. \cite[ Theo. 10.9.1.]{15} and \cite{9}.}
\begin{prop}\label{exten}
$F^*: H_f\rightarrow H_{f\circ F}$ is an injective map of $\mathbb{C}[s]$-modules.
Pull-backs of $\eta_{\beta}$ for all $\beta\in I_f$ are independent in $H_{f\circ F}$ under the map $F^*$ and can be extended to a basis of $H_{f\circ F}$. 
\end{prop}
\begin{proof}
 By using Theorem \ref{15:39:1} we can write   
\begin{eqnarray*}
F^*(\eta_{\beta})=R^{\beta_1}S^{\beta_2}(RdS-SdR)=\sum_{0\leq i,j\leq an+n-1}P_{ij}(s')\bar{\eta}_{ij},
\end{eqnarray*}
where $\eta_{\beta}=x^{\beta_1}y^{\beta_2}(ydx-xdy)$.
We are going to show that the functions $P_{ij}(s')\in\mathbb{C}[s']$ are constant, where $s'=f\circ F$.
According to Theorem \ref{15:39:1} all of the coefficients $P_{ij}(s')$ are of degree $\frac{n(i+j)+2n}{n(a+1)}-A_{\beta}$, but $n(a+1)>n(i+j)+2n$ which implies $P_{ij}$ is constant for all $i,j$.
 It is clear that the map $F^*: H_f\rightarrow H_{f\circ F}$ is injective, because if $F^*(\sum h_\beta \eta_{\beta})=0$, then $\sum_{\beta} F^*(h_\beta) F^*(\eta_{\beta})=\sum_{\beta} [F^*(h_\beta)\sum_{ij} c_{ij}^\beta\bar{\eta}_{ij}]=\sum_{ij}[\sum_{\beta}F^*(h_\beta)c_{ij}^\beta]\bar{\eta}_{ij}=0$, $c_{ij}^\beta$ are constant, and since $\bar{\eta}_{ij}$ forms a basis of $H_{f\circ F}$ thus $\sum_{\beta}F^*(h_{\beta})c_{ij}^\beta=0$ for all $ij$. This implies that $h_{\beta}=0$ for all $\beta\in I_{f}$. Therefore $F^*(\eta_{\beta})$ are linear independent because $F^*$ is injective.
\\
In other words, according to above conditions $F^*(\eta_{\beta})$ for all $\beta$ can extend to a basis of $H_{f\circ F}$.  
\end{proof}
\subsection{Relatively Exact 1-form}
\begin{defi}
Let $\mathcal{F}$ be a foliation in $\mathbb{C}^{2}$. A meromorphic 1-form $\omega$
is called relatively exact modulo $\mathcal{F}$, if the restriction of $\omega$ to each leaf $L$ of $\mathcal{F}$ is exact, i.e. there is a meromorphic function $f$ on $L$ so that $\omega|_L =df$.
\end{defi} 

One can check that a meromorphic 1-form $\omega$ is relatively exact modulo $\mathcal{F}$, if 
$$
\int_{\delta}\omega =0,
$$ 
for all closed cycles in the leaves of $\mathcal{F}$, where this integral is well-defined.
\begin{prop}\label{12:43}
Let $f\in\mathbb{C}[x,y]$ be a polynomial with isolated critical points, and suppose that for every $t\in\mathbb{C}$ the fiber $f^{-1}(t)\subset\mathbb{C}^2$ is connected. Every relatively exact polynomial 1-form $\omega$ modulo $\mathcal{F}(df)$ on $\mathbb{C}^2$ is of the form 
\[\omega=dg+pdf\]
where $p, g$ are polynomials with degrees less than or equal to $deg(\omega)-deg(df)$ and $deg(\omega)$, respectively.
\end{prop}
See e.g. (\cite[Theo. 1.2.]{9}).

The above theorem is not true if we assume that some $f$-fibers are disconnected. P. Bonnet in  \cite{2} gives the example $f=x(xy+1)$ in $\mathbb{C}^2$ having the disconnected fiber $f=0$. The 1-form $y^2dx+xydy$ are relatively exact modulo $df$ but it is not an exact form modulo $df$. Note that the hypothesis of connectedness of $f$-fibers is valid for a generic polynomial, it means that any generic polynomial has connected fibers.
\subsection{Computing the Tangent Cone }
Let $F$ be a morphism from $\mathbb{C}^2$ into itself and $f$ be a polynomial of degree $a+1$ that are defined in (\ref{12:35}) and (\ref{12:51}) respectively. Consider
the deformation $\mathcal{F}_{\epsilon}=\mathcal{F}(\omega_{\epsilon})$ of $\mathcal{F}(d(f\circ F))$, where  
\[
\omega_{\epsilon}=d(f\circ F)+\epsilon^k \omega_k+\epsilon^{k+1}\omega_{k+1}+\dots\ \ \  , \ \ deg(\omega_j)\leq an+n-1.
\]
Then from the equality (\ref{15/09}) we have the following Theorem:
\begin{theo}\label{23:24} 
There exists a polynomial differential 1-form $\alpha_1$ with $deg(\alpha_1)\leq deg(\omega_k)$ and a polynomial $K\in \mathbb{C}[x,y]_{\leq {a+1}}$ such that $$\omega_k=F^*(\alpha_1)+dK,$$ where $F:\mathbb{C}^2\rightarrow\mathbb{C}^2$ is defined by $(x,y)\mapsto (R,S)$ as in (\ref{12:35}).
\end{theo}
\begin{proof}
For a regular value $b$ of the function $f\circ F$, by using Theorem \ref{1:29:1}, $\int_{\delta}\omega_k=0$ for all $\delta\in ker(F_*)$. This implies that the linear map  
\begin{eqnarray*}
F_*:H_1((f\circ F)^{-1}(b),\mathbb{Z})\rightarrow H_1((f)^{-1}(b),\mathbb{Z}),
\end{eqnarray*}
is surjective.
Then 
\begin{eqnarray*}
F_b^*=H_{dR}^1(f^{-1}(b))\rightarrow H_{dR}^1((f\circ F)^{-1}(b))
\end{eqnarray*}
 is injective.Therefore the linear map 
\begin{eqnarray*}
H_1(f^{-1}(b),\mathbb{Z})\rightarrow \mathbb{C} & defined \ \ by & 
\delta\mapsto\int_{\gamma}\omega_k
\end{eqnarray*}
for an element $\gamma\in F_*^{-1}(\delta)$, is well defined. By duality of de Rham cohomology and singular homology there is a differential form $\alpha_b$ in $f^{-1}(b)$ such that 
\begin{eqnarray*}
\int_{\gamma}\omega_k=\int_{\delta}\alpha_b \ \ for \ all \ \gamma\in F^{-1}(\delta).
\end{eqnarray*} 
By using Atiyah-Hodge theorem ( see e.g. \cite{19}) the form $\alpha_b$ can be taken algebraically. All these $\alpha_b$'s give us a holomorphic global section $\alpha$ of cohomology bundle of $f$ outside the critical values of $f$ such that  
\begin{eqnarray*}
\alpha|f^{-1}(t)=\alpha_t,\ \ \alpha_t\in H_{dR}^1(f^{-1}(t)) &where& t\in\mathbb{C}\setminus \{c_1,\dots,c_{a^2}\}.
\end{eqnarray*} 
We are going to show that it is a holomorphic global section in the whole $\mathbb{C}$.
\\
By the Theorem \ref{15:39:1} we can write 
\begin{eqnarray*}
\alpha=\sum_{ \beta} h_{\beta}\eta_{\beta} &where& h_{\beta} \ s\ \ are \ \ holomorphic \ \ in \ \ \mathbb{C}\setminus\{c_1,\dots,c_{a^2}\} \ , \ \beta=(i,j).
\end{eqnarray*} 
The Period matrix 
\[
\begin{bmatrix}
  \int_{\delta_{c_k}}\eta_{\beta} 
\end{bmatrix}_{\mu\times\mu}
\]
where $\delta_c$ is a vanishing cycle and $c\in C$, is invertible, where $\mu$ is the rank of $H_1(f^{-1}(b),\mathbb{Z})$, (see e.g. \cite[Prop. 26.44]{11} ). Therefore, the $h_{\beta}$'s coefficients are meromorphic functions on $t$, because 
\[
\begin{bmatrix}
  h_{\beta_1} \\
   \vdots  \\
   h_{\beta_{\mu}}
\end{bmatrix}
=
\begin{bmatrix}
  \int_{\delta_{c_k}}\eta_{\beta} 
\end{bmatrix}_{\mu\times\mu}^{-1}
\begin{bmatrix}
  \int_{\delta_{c_1}}\alpha \\
   \vdots  \\
    \int_{\delta_{c_{\mu}}}\alpha
\end{bmatrix},
\]
and by using \cite[Ch 10,Theo. 10.7 ]{1} each integral $||\int_{\delta_{c_k}} \alpha||\leq\ const\ ||t-c_k||^{-N} $ for a natural number $N$ and  $t$ close to singular value $c_k$. Thus all the elements of the matrices on the right hand side of the equality have finite growth at critical values. 
This implies that, there is a polynomial $P(s)\in \mathbb{C}[s]$ such that $P.\alpha$ is a holomorphic form. We can write $P.\alpha =\sum_{\beta} {h'}_{\beta}\eta_{\beta}$, then $F^*(P)\omega_k-F^*(P.\alpha)=0$ in $H_{f\circ F}$. According to Proposition \ref{exten} the set of $F^*(\eta_{\beta})$ for all $\beta$ can be extended to a basis of $H_{f\circ F}$.
Therefore, we have 
\begin{eqnarray}\label{7:26}
F^*(P)\omega_k=\sum_{\beta}F^*(P). {b}_{\beta}(s)F^*(\eta_{\beta})+\sum_{\sigma}F^*(P)a_{\sigma}\tilde{\eta}_{\sigma}.
\end{eqnarray}
Since each element of $H_{f\circ F}$ can be written uniquely as a linear combination of the elements on this basis, then $a_{\sigma}=0$ for all $\sigma$. In other words, $F^*(P). {b}_{\beta}=F^*( {h'}_{\beta})$, hence $P$ divides ${h'}_{\beta}$. This implies that $\alpha$ is a holomorphic 1-form. By Theorem 
\ref{15:39}, the degree of $h_{\beta}$ in the equation (\ref{7:26}) is less than or equal to $ \frac{deg(\omega_k)}{as+s}-A_{\beta}<1$, hence  $h_{\beta}$ are constant for all $\beta$. To find the form of $\omega_k$ we use the Proposition \ref{12:43}
and we conclude that $\int_{F^*(\delta)}\alpha=\int_{\delta}\omega_k$ for all cycles $\delta$ in the fibers of $f\circ F$.
This implies that $\omega_k-F^*(\alpha)$ is relatively exact modulo $\mathcal{F}(d(f\circ F))$. Here by Proposition \ref{12:43} there are polynomials $K$ and $A$ such that $\omega_k-F^*(\alpha)=dK+Ad(f\circ F)$.  
The fact that $deg(\omega_k-F^*(\alpha))\leq deg(f\circ F)-1$ implies $A\equiv 0$, so we get our desired equality. 
\end{proof} 
The proof of the main theorem still is not finished. We have to prove that the polynomial $K$ in the Theorem \ref{23:24} is of the form of the equation \eqref{23:25}. 
In order to reach this goal, we need to compute higher order Melnikov functions. 
This will be done in the next section. 
\section{Higher order Melnikov functions }
L. Gavrilov in \cite{8} has shown that the higher order Melnikov functions can be expressed in terms of iterated integrals.
Basic properties of iterated integrals are established by A. N. Parsin in 1969, and a systematic approach to de Rham cohomology type theorems for iterated integrals was made by K. T. Chen around 1977. For further details on iterated integrals see \cite[Ch 6]{17} or \cite{18}. In order to prove  Theorem \ref{Main theorem} we have to compute the higher order Melnikov function, and show that the polynomial $K$ in Theorem \ref{23:24} has the form of the equation \eqref{23:25}. 

Let $\gamma:[0,1]\rightarrow \mathbb{C}^2$ be a piecewise smooth path on $\mathbb{C}^2$. Let $\omega _1 ,\omega_2,\dots ,\omega_n$ be smooth
1-forms on $\mathbb{C}^2$, $ \gamma^{*}(\omega_i)=f_i(t)dt$
for the pull-back of the forms $\omega_i$ to the interval [0,1].
Recall that the ordinary line integral given by
$$
\int_{\gamma}\omega_1 =\int_{[0,1]}\gamma^{*}(\omega_1) =\int_0^1 f_1(t_1)dt
$$
does not depend on the choice of parametrization of $\gamma$.
\begin{defi}
Iterated integral of $\omega_1,\omega_2,\dots , \omega_n $ along the path $\gamma$ is defined by
$$
\int_{\gamma}\omega_1.\omega_2\dots\omega_n =\int_{0\leq t_1 \leq \dots \leq t_n \leq 1}(f_1(t)dt_1 \dots f_n(t_n)dt_n).
$$
We have for instance 
\[
\int_{\gamma}\omega_1\cdot\omega_2=\int_{\gamma}\omega_1\int_{\gamma(0)}^{\gamma(t)}\omega_2.
\]
\end{defi}
Let $\omega_2=dK$ be an exact form so we have 
\[
\int_{l}\omega_1\cdot dK=\int_{l}\omega_1\int_{l(0)}^{l(t)}dK=\int_{l}\omega_1\left(K(l(t))-K(l(0))\right)=\int_{l}K\omega_1-K(l(0))\int_{l}\omega_1.
\]
Let us consider the deformation 
$$
\mathcal{F}_{\epsilon}:d(f\circ F)+\omega_k\epsilon^k+\omega_{k+1}\epsilon^{k+1}+\dots \ \ \ deg(\omega_i)\leq n(a+1)-1
$$	
of $d(f\circ F)$. The deformed holonomy along the path $\delta_t$ in $\Sigma$ is
\[
h_{\epsilon}(t)-t=M_1(t)\epsilon+\dots+M_k(t)\epsilon^k+\dots+M_{2k}(t)\epsilon^{2k}+\dots.
\]
Since $\omega_i=0$ where $0<i<k-1$, then $M_1=M_2=\dots=M_{k-1}=0$. 
By using Theorem 3.2 in \cite{17} (Higher order approximation), we conclude that 
$M_i(t)=-\int_{\delta_t}\omega_i$ where $k\leq i<2k$, and also 
\[
M_{2k}(t):=-\int_{\delta_t}(\omega_k .(\frac{d\omega_k}{d(f\circ F)})+\omega_{2k}).
\]
Note that the vector $W$ in equation \eqref{13:59:03} when $\omega=d(f\circ F)$  is of the form 
\begin{eqnarray}\label{23:25}
W=d(R_1\frac{\partial f}{\partial x}(R,S)+S_1\frac{\partial f}{\partial y}(R,S))+F^{*}(\alpha_1).
\end{eqnarray}
\begin{lem}\label{13:17}
The polynomial $K$ in Theorem \ref{23:24} is of the form 
$$
K=R_1\frac{\partial f}{\partial x}(R,S)+S_1\frac{\partial f}{\partial y}(R,S), 
$$
where $ R_1 ,S_1 \in \mathbb{C}[x,y]_{\leq n}$. 
\end{lem}

\begin{proof}
\begin{eqnarray*}
\int_{\delta_t}\omega_k .(\frac{d\omega_k}{d(f\circ F)})&=& \int_{\delta_t} (F^{*}\alpha_1 +dK ).F^{*}(\frac{d\alpha_1}{df})\\
 &=& \int_{\delta_t}F^{*}(\alpha_1 .\frac{d\alpha_1}{df}) + \int_{\delta_t}(dK).F^{*}(\frac{d\alpha_1}{df}) \\
 &=& \int_{F_{*}(\delta_t)}\alpha_1.\frac{d\alpha_1}{df} + \int_{\delta_t}K F^{*}(\frac{d\alpha_1}{df}) - K(p_t)\int_{\delta_t}F^{*}(\frac{d\alpha_1}{df}) \\
 &=& \int_{\delta_t}K F^{*}(\frac{d\alpha_1}{df}).
\end{eqnarray*}
Here $ p_t$ is a point belonging to the cycle $\delta_t$. The last equality $\int_{F_{*}(\delta_t)}\alpha_1.(\frac{d\alpha_1}{df})$ follows form the fact that $F^*(\delta)$ is homotopic to a constant path. The equality $M_{2k}(t)=0$ and a similar argument as in Theorem \ref{23:24} implies that
$$
\omega_{2k}+KF^{*}(\frac{d\alpha_1}{df})=F^{*}(\alpha_2) + dK_2 + A_jd(f\circ F),
$$
and therefore,
$$
KF^{*}(d\alpha_1) = -\omega_{2k}\wedge d(f\circ F)+ F^{*}(\alpha_2)\wedge d(f\circ F)+ dK_2\wedge d(f\circ F).
$$\\
Since $d\alpha_1$ is a 2-form like $h(x,y)dx\wedge dy$, 
$$F^{*}(h(x,y)dx\wedge dy)=F^{*}(h).F^{*}(dx\wedge dy)=(h\circ F)(R_x S_y-R_y S_x)dx\wedge dy. $$
This implies that $Kh(F)(R_x S_y-R_y S_x)$ belongs to the ideal
$$I=<R_x f_x (R,S)+R_y f_y(R,S) \ , \ S_x f_x(R,S)+S_y f_y(R,S)>.$$
Consider the ideal $I_1=<f_x(R,S) \ , \ f_y(R,S)>$ and $J=<h(F).(R_x S_y-R_y S_x)>$. Then it is clear that
$I\subset I_1$ and $I_1.J\subset I$, hence $$I_1.J\subseteq I\cap J \subseteq I_1\cap J.$$
Since  $\int_{D_\delta}d\omega_k=\int_{D_\delta} F^*(d\alpha_1)$, where $D_\delta$ is a real disk with boundary $\delta$ which is vanishing cycle around the critical point $p$. If $h$ vanishes at a point in $V(f_x)\cap V(f_y)$, then $\int_{D_\delta} F^*(d\alpha_1)=0$. Let $\delta$ be a pull back vanishing cycle, since the action of monodromy group $\pi_1(\mathbb{C}\setminus C,b)$ on $\delta$ generates all $H_1((f\circ F)^{-1}(b),\mathbb{Z})$, thus  $\omega_k$ will be relatively exact form which in general is not. We can assume that the curve $V(h)$ does not pass any critical points of $f$. By our hypothesis $F$ and $f$ are generic, thus we have $V(J)\cap V(I_1)=\emptyset $. This means that $J+I_1=\mathbb{C}[x,y]$ thus we have
\[
I_1.J=I_1\cap J \ \ \ \ \Rightarrow \ \ \ I\cap J = I_1 .J
\]
which states that $K\in<f_x(R,S) \ , \ f_y(R,S)>$. Therefore we get the result
$$K=R_1\frac{\partial f}{\partial x}(R,S)+S_1 \frac{\partial f}{\partial y}(R,S) \ \ \ \ where \ \ \ R_1 \,\ S_1 \in\mathbb{C}[x,y]_{\leq n} .$$
\end{proof}
\begin{coro}\label{12:18}
The tangent cone of $\mathcal{M}(as+s-1)$ at the point $\mathcal{F}_0:=\mathcal{F}(d(f\circ F))$ is as the following:
$$TC_{\mathcal{F}_0}\mathcal{M}(as+s-1)=T_{\mathcal{F}_0}\mathcal{P}(a,s)\cup T_{\mathcal{F}_0}\mathcal{H}(as+s-1).$$
\end{coro}
\subsection{Proof of main theorem }\label{18:51:03}
Consider a germ of an analytic variety ($X$,0) in ($\mathbb{C}^n$,0).  
The holomorphic curve $\gamma :(\mathbb{C},0)\rightarrow (X,0)$  has the Taylor expansion $\gamma=\omega\epsilon^l+\omega'\epsilon^{l+1}+\dots$ , $\omega, \omega',\dots \in \mathbb{C}^n$. Let $T_l$ be the set of all $\omega$. The tangent cone $TC_0X$ of $X$ at 0 is $TC_0X=\cup_{l=1}^{\infty} T_l$. \\

The variety $\mathcal{P}(n,a)$ is parametrized by 

\begin{equation}
\tau :\mathcal{P}_n\times\mathcal{P}_n\times\mathcal{P}_a\times\mathcal{P}_a\rightarrow\mathcal{F}(d),\ \ \ \  \ d=na+n-1
\end{equation}
$$	
\tau(R,S,P,Q)=P(R,S)dS-Q(R,S)dR , n\geq2,
$$ 
and so it is irreducible.
\begin{proof}\emph{of Theorem \ref{1:28}:}
 Let $\mathcal{F}_0:=\mathcal{F}(d(f\circ F))$ where $f$ and $F$ defined in \eqref{12:35} and \eqref{12:51} respectively.
 For the proof of our main theorem, it is enough to show that $X:=(\mathcal{P}(n,a),\mathcal{F}_0)$ is an irreducible component of $(\mathcal{M}(d),\mathcal{F}_0)$.
According to Corollary \ref{12:18} we have:
\begin{equation}\label{13:44}
 TC_{\mathcal{F}_0}\mathcal{M}(d)= TC_{\mathcal{F}_0}\mathcal{P}(n,a)\cup TC_{\mathcal{F}_0} H(d).
\end{equation}
Let $X'$ be an irreducible component of $(\mathcal{M}(d),\mathcal{F}_0)$ such that $X\subset X'$. If $TC_{\mathcal{F}_0}X\subset Y$, where $Y$ is the irreducible component of $TC_{\mathcal{F}_0}X'$, then it is a subset of $TC_{\mathcal{F}_0}X$, because the equality (\ref{13:44}) is a decomposition of $TC_{\mathcal{F}_0}\mathcal{M}(d)$ 
to irreducible component. This implies that $Y=TC_{\mathcal{F}_0}X$. The dimension of $Y\subset TC_{\mathcal{F}_0}X'$ is equal to the dimension $X$, so $dim(X')=dim(X)$. Therefore $X=X'$ because $X\subset X'$ and $X$ and $ X'$ are irreducible algebraic sets and they have the same dimension.
\end{proof}
\section{Limit cycles}
Consider a real planar 1-form $\omega=P(x,y)dy-Q(x,y)dx$ where 
$P$ and $Q$ are polynomials of degree less than or equal to $d$. Let $\mathcal{F}$ be the foliation induced by the 1-form $\omega$.
\begin{defi}
A closed trajectory which is a limit set of some trajectories of a real foliation $\mathcal{F}$ is called a limit cycle.
\end{defi}
The \emph{Hilbert number} $H_d$ is the maximum possible number of limit cycles of a real foliation $\mathcal{F}(\omega)$. It is still unsolved whether $H_d$ is finite, even for the simple case $d=2$. It is known that $H_d\geq k.d^2$ for some constant $k$, but in 1995, C.J Christopher and N.G. Lloyd found a strong lower bound $d^2log\ d$  for the Hilbert numbers, see \cite{4}.  
   
Let $X$ be an irreducible component of $\mathcal{M}(d)$. 
Let $p$ be a real center singularity of a real foliation $\mathcal{F}\in X-sing(\mathcal{M}(d))$. By real foliation we mean the  equation of the foliation has real coefficient. 
Let $\delta_t,t\in (\mathbb{R},0)$ be a family of real vanishing cycles around $p$.
Roughly speaking, the \emph{cyclicity} of $\delta_0$ is the maximum number of
limit cycles appearing near $\delta_0$ after a deformation of $\mathcal{F}$ in
$\mathcal{F}(d)$.
The cyclicity of $\delta_0$ in a deformation of $\mathcal{F}$ inside $\mathcal{F}(d)$ 
is greater
than $codim_{\mathcal{F}(d)}(X)-1$. 
One can find the exact definition of 
cyclicity and the proof of this fact in \cite{20}. Yu. Ilyashenko in \cite{10} shows that  
$$codim_{\mathcal{F}(d)}(\mathcal{H}(d))-1=\frac{(d+2)(d-1)}{2}-1.$$ The best upper bound for the cyclicity of a vanishing cycle of a Hamiltonian equation is the P. Mardesic's result $\frac{d^4+d^2-2}{2}$ in \cite{13}. H. Movasati in \cite{16}, shows that the cyclicity of $\delta_0$ of a logarithmic foliation $\mathcal{F}(f\sum_{i=1}^s\lambda_i\frac{df_i}{f_i})\in\mathcal{L}(d_1,\dots,d_s)$ is not less than
\[
(d+1)(d+2)-\sum_{i=1}^{s}(\frac{(d_i+1)(d_i+2)}{2})-1.
\]
This lower bound reaches to maximum when $d_i=1 , s=d+1 , i=1,\dots,s$. In this case the cyclicity of $\delta$ is not less than $d^2-1$.  
\begin{prop}
Let $n>1$ and $d:=an+n-1$, the cyclicity of $\delta_0$ in a deformation of $\mathcal{F}$ in $\mathcal{F}(d)$ 
is not less than 
\[
C:=(d+1)(d+2)-((n+1)(n+2)+(\frac{d+1}{n})(\frac{d+1}{n}+1))-1.
\]
\end{prop}
By considering  $d+1=(a+1)n$  as a fixed value, when $n\leq a+1$ and in addition the distance of $a+1$ and $n$  is minimum, then $(n+1)(n+2)+(\frac{d+1}{n})(\frac{d+1}{n}+1)$ will be minimum. This minimization will lead to maximizing of the cyclicity. If $n$ and $(a+1)$ are near to $\sqrt{d+1}$ then the cyclicity $C$ is close to $d^2+d-4\sqrt{d+1}-3$. If $a+1=p$ and $n=q$ , where $p,q$ are primes and $p>q$ then $C=(pq)^2+pq-q^2-3q-p^2-p-3$, for instance when $q=2$ we have $C=3p^2+p-13$.   
  \section{Picard-Lefschetz Theory}
In this section, we intend to study the topology of regular fibers of polynomial functions composition that are transversal to infinity with two variables of the form $u(x)+v(y)$. See \cite[Ch 7]{15} for the general cases. For this, first we study the topology of polynomials in one variable.  The main idea of this section is to understand the intersection number between two vanishing cycles and the action of monodromy group on a vanishing cycle in the case of pull-back of cycles under a morphism.

Suppose that a function $f$ has only non-degenerate critical points with the finite set of critical values $C$ labeled by $c_1,c_2,\dots,c_s$. 
\begin{defi} 
A distinguished system of paths related to $f$ is the system of smooth paths in $\mathbb{C}$, starting at the regular point $b\in \mathbb{C}\setminus C$     
and ending at a point in $C$ such that 
\begin{enumerate}
\item The paths have no self-intersections;
\item Different paths meet only at their
common point $b$.
\end{enumerate}
\end{defi}
Consider a small ball $U_p$ in $\mathbb{C}^m$ with the center at the Morse critical point $p$. Let the value $b$ be very close to $c:=f(p)$ but not equal to it. Let $\alpha:=[0,1]\rightarrow f(U_p)$ be a path that starts at $b$, ends at $c$ and does not pass through any other critical value of $f$. By the Morse lemma, there is a local coordinate system $x_1,\dots,x_m$ in a neighborhood of $p$ such that the function $f$ can be written in the form $f(x_1,\dots,x_m)=c+\sum_jx_j^2$. Consider the sphere $S_t:=\{x\in f^{-1}(\alpha(t))|Im(x_j(x))=0\}\cap U_p$. Whenever $t$ tends to 1, then $S_t$ tends to $p$.
\begin{defi} 
The homology class $\delta\in H_{m-1}(f^{-1}(b),\mathbb{Z})$ defined by the sphere $S_0$, in the nonsingular fiber $f^{-1}(b)$, is called a vanishing cycle along the path $\alpha$ or just a vanishing cycle.  
\end{defi}
\begin{theo}(see e.g \cite{1})
The collection of the vanishing cycles along all paths of a distinguished system of paths, forms a basis of the group $H_{m-1}(f^{-1}(b),\mathbb{Z})$.   
\end{theo}
\begin{defi}
Let $\lambda_c$ be a path of a distinguished system, and $\lambda$ be a loop in $\pi_1(\mathbb{C}\setminus C,b)$ such that 
\begin{enumerate*}[label={\alph*)},font={\color{red!50!black}\bfseries}]
\item it turns once anti-clockwise around the critical point $c$, and
\item the closure of the interior of $\lambda$ contains the path $\lambda_c$ and does not contain any other point of $C$.
\end{enumerate*} Then the loop $\lambda$ is called a simple loop, corresponding to $\lambda_c$.    
\end{defi}
The paths that are homotopic to a simple loop $\lambda $  give a class of homotopic homeomorphism maps  $\{h_{\lambda}:f^{-1}(b)\rightarrow f^{-1}(b)\}$. This class defines  a unique well-defined map \[h_{\lambda}:H_{m-1}(f^{-1}(b),\mathbb{Z})\rightarrow H_{m-1}(f^{-1}(b),\mathbb{Z}).\]
\begin{defi}
For a regular value $b$ of $f$, we have 
\begin{gather*} 
h:\pi_1(\mathbb{C}\setminus C,b)\times H_{m-1}(f^{-1}(b),\mathbb{Z})\rightarrow H_{m-1}(f^{-1}(b),\mathbb{Z})\\
h(\lambda,\cdot)=h_{\lambda}(\cdot).
\end{gather*}
The image of $\pi_1(\mathbb{C}\setminus C,b)$ in $Aut(H_{m-1}(f^{-1}(b),\mathbb{Z}))$, is called the monodromy group and its action $h$ is called the action of the monodromy group on the homology group of $f^{-1}(b)$.
\end{defi}
\textbf{Picard-Lefschetz formula:}
Let $\lambda$ be the simple loop around the critical value $c$, the action of monodromy $h_{\lambda}$ on a cycle $\delta\in H_{m-1}(f^{-1}(b),\mathbb{Z})$ is given by 
\begin{equation}
h_{\lambda}(\delta)=\delta+\sum_j(-1)^{\frac{m(m-1)}{2}}<\delta,\delta_j>\delta_j
\end{equation}  
where $j$ runs through all the vanishing cycles around the singularities with value $c$,  and $<\cdot,\cdot>$ denotes the intersection number of two cycles in $f^{-1}(b)$. 

From now on, for simplicity we will denote $h_{\lambda}$ by $\lambda$. 
\subsection{Picard-Lefschetz theory in dimension zero}
Let $f(x)\in\mathbb{C}[x]$ be a polynomial of degree $d$ with real roots $r_i$ where $i=1,\dots,d$, and $d-1$ critical values.  
\begin{theo}\label{5/8/16}
For the regular value $b=0$, we have the following:
\begin{itemize}
	\item $H_0(f^{-1}(b),\mathbb{Z})$ is generated by $\delta_i=[r_i]-[r_{i+1}]$ for $i=1,\dots,d-1$;
	\item The intersection matrix for $H_0(f^{-1}(b),\mathbb{Z})$ with respect to this basis is 
\[
 <\delta_i,\delta_j> = 
  \begin{cases} 
   2 & \text{if } i=j \\
   -1& \text{if}\ |i-j|=1\\
    0& \text{if}\  |i-j|>1.
  \end{cases}
\]
(See e.g. \cite{15} or \cite{7}.)
\end{itemize}
\end{theo}
\begin{lem}\label{23:45}
For any two vanishing cycles $\delta_i ,\ \delta_j \in H_0(f^{-1}(b),\mathbb{Z})$ there is a monodromy $\lambda$ such that $\lambda(\delta_i)=\delta_j$.
\end{lem}
Consider
\begin{eqnarray}
R(x)=\prod_{i=1}^{n}(x-t_i)
\end{eqnarray}
where $(t_i\in\mathbb{R}^+ \ , \ t_i<t_{i+1})$ such that $R$ has $n-1$ distinct critical values. 
Let us define
\begin{eqnarray}
g(x)=\prod_{i=1}^{a+1} (x-s_i),&where& a>1.
\end{eqnarray}
Assume that it meets the following conditions 
\begin{enumerate}
 \item $s_i$s are positive real numbers and $ s_i\neq s_j$ for all $i$ and $j$,
 \item The function $g$ has distinct critical values  also $R(x)=s_i$ has $n$ real roots. $s_i$s are in an interval $I$ such that $g^{-1}(I)$ is a union of $n$ intervals,   
 \item The function $g\circ R(x)$ has only non-degenerate critical points.
\end{enumerate} 
\begin{nota}\label{12:38}
Let us denote by  $C\cup\tilde{C}$ the set of critical values of $g\circ R$, where $C$ is the set of the critical values of $g$, and $\tilde{C}$ is the image of the set of the critical points of $R$ under $g\circ R$. All the critical points of $g$ and $R$ are real. Therefore $C=\{c_i=g(p_i)| p_i\in V(g_x)\ , \ p_1<p_2<\dots<p_a\}$ when $n$ is odd, and $C=\{c_{a+1-i}=g(p_i)| p_i\in V(g_x)\ , \ p_1<p_2<\dots<p_a\}$ when $n$ is even. Also the order of $\tilde{C}$ is as usual  $\{\tilde{c}_{a+j}=g\circ R(q_j) \ | \ q_j\in V(R_x) \ , \ q_1<q_2<\dots<q_{n-1}\}$. 
\end{nota}
Take the distinguished system of paths related to the function $g\circ R$ such that all the paths are in the upper half plane. 
Let $\gamma_c$ be the vanishing cycle along the path $\lambda_c$ of the fiber $g^{-1}(0)$. Therefore $R_*^{-1}(\gamma_c)=\{\delta_c^i|i=1,\dots,n\}$ is the set of vanishing cycles along the path $\lambda_c$ of the fiber $(g\circ R)^{-1}(0)$. 
\begin{theo}\label{7:07}
If $b=0$, then the 
zero homology group of $(g\circ R)^{-1}(b)$ is generated by
\begin{equation}\label{7:07:07}
\delta_c^i,\delta_{\tilde{c}}\ \ where\ \ c\in C \ , \ \tilde{c}\in\tilde{C} \ , \ i=1,\dots,n.
\end{equation} 
\end{theo}
\begin{theo}\label{7:15}
For the regular value $b=0$ of $g\circ R$,
the intersection form of $H_0((g\circ R)^{-1}(b),\mathbb{Z})$ with respect to the basis
 in Theorem \ref{7:07} is as the following:

\[
<\delta_i^j\ , \ \delta_{i'}^{j'}>=
            \begin{cases}
			2 & \text{if }{(i=i',j=j')}\\ 
			-1 & \text{if }
			\begin{cases}
			 (j=j',i'=i+1 ,i<a)\\
			\vee({j'}=0,i=a,{i'}=a+2k-1,j=k+(\frac{(1+(-1)^{n})}{2}))\\
			\vee(j'=0,i=1,i'=a+2k ,j=2k+(\frac{(1+(-1)^{n+1})}{2}))
			 \end{cases}\\
			1 & \text{if }
			\begin{cases}
			(j'=0,i=a,i'=a+2k-1,j=k+(\frac{1+(-1)^{n+1}}{2}))\\ 
			\vee(j'=0,i=1,i'=a+2k,j=2k+(\frac{1+(-1)^n}{2}))
			\end{cases}
			\\ 
            0 & \text{if }
            \begin{cases}
            (j=j',|i'-i|>1) \vee ( j=j'\ ,\ i,i'>a ,  i\neq i' )\\
            \vee(i= i',j\neq j').
			\end{cases}
		\end{cases}
\]
Here, by $\delta_{i}^j$ and $\delta_{a+k}^0$  we mean $\delta_{c_i}^j$ and $\delta_{\tilde{c}_{a+k}}$ respectively, and $k\in\mathbb{N}$.  
\end{theo}
\begin{proof}
Assume that $n$ is odd. Consider the vanishing cycle $\delta_{a+2k-1}^0:=[r]-[r']$ , where $r$ and $r'$ are two consecutive roots of $g\circ R$. Thus $R(r)=R(r')=s_{a+1}$ is a root of $g$. Let $s_{a}$ and $s_{a+1}$ be two consecutive roots of $g$ and also let $R^{-1}(s_{a})=\{l_j| l_1<l_2<\dots<l_n\}$ such that $l_k<r'<r<l_{k+1}$ be four consecutive roots of $g\circ R$. Two vanishing cycles on the basis \eqref{7:07:07} are $\delta_{a}^{2k-1}=[r']-[l_k]\ $ and $ \ \delta_{a}^{2k}=[r]-[l_{k+1}]$; therefore
\[
<\delta_{a+2k-1}^0,\delta_{a}^{2k-1}>=-<\delta_{a+2k-1}^0,\delta_{a}^{2k}>=-1
\]
\[
<[r]-[r'],[r']-[l_k]>=-<[r]-[r'],[r]-[l_{k+1}]>=-1.
\]
For the vanishing cycle $\delta_{a+2k}^0:=[r_2]-[r_1] $ where $r_1$ are $r_2$ are two consecutive roots of $g\circ R$, we have $R_*(\delta_{a+2k})=0$ and $R(r_1)=R(r_2)=s_1$ where $s_1$ is a root of $g$. For the root $s_2$ of $g$ let $R^{-1}(s_2)=\{ l_j| l_1< l_2<\dots<l_n\}$, so we have $l_{2k}<r_1<r_2<l_{2k+1}$ as consecutive roots of 
$g\circ R$. Therefore $\delta_1^{2k}=[r_1]-[l_{2k}]$ and $\delta_1^{2k+1}=[l_{2k+1}]-[r_{2}]$ are vanishing cycles and they are in the basis so we have 
\[
<\delta_{a+2k}^0,\delta_{1}^{2k}>=-<\delta_{a+2k}^0,\delta_{1}^{2k+1}>=1
\]
\[
<[r_2]-[r_1],[l_{2k}]-[r_1]>=-<[r_2]-[r_1],[l_{2k+1}]-[r_2]>=1.
\]
 Also the above procedure can work for an even number $n$ but 
only by changing the order of $C$. 
The function $R$ induces the surjective morphism 
\[
R_*:H_0((g\circ R)^{-1}(b),\mathbb{Z})\rightarrow H_0(g^{-1}(b),\mathbb{Z}).
\]
If $\gamma_{c_i},\gamma_{c_{i+1}}\in  H_0(g^{-1}(b),\mathbb{Z})$ are two vanishing cycles where $c_i,c_{i+1}\in C$, then each set $R_*^{-1}(\gamma_{c_m})=\{\delta_{c_m}^j|j=1,2,\dots , n\}$, $m=i,i+1$, contains $n$ separated vanishing cycles. For each cycle $\delta_{c_i}^j$ there is exactly one cycle $\delta_{c_{i+1}}^j$ such that $<\delta_{c_i}^j,\delta_{c_{i+1}}^j>=-1$. By definition of the functions $R$ and $g$ and also Theorem \ref{5/8/16}, we can have the other equalities.    
\end{proof}
\begin{defi}
The Dynkin diagram of a polynomial with only non-degenerate critical points is a graph defined in the following way: Its vertices are in one-to-one correspondence with a distinguished basis of vanishing cycles $\delta_i$, $i=1,2,...,\mu=d-1$. The i-th and j-th vertices of the graph are joined with an edge of multiplicity $<\delta_i,\delta_j>$.
 The intersection indexes $(-1)^n$ are depicted by dash lines, where $n$ is the dimension of vanishing cycles.  
\end{defi}
For an illustration of the Dynkin diagram $g\circ R$ with respect to this basis (when $n$ is odd) see Figure \ref{17;21:21}.

\begin{figure}[h]
\centering
\includegraphics[width=8cm]{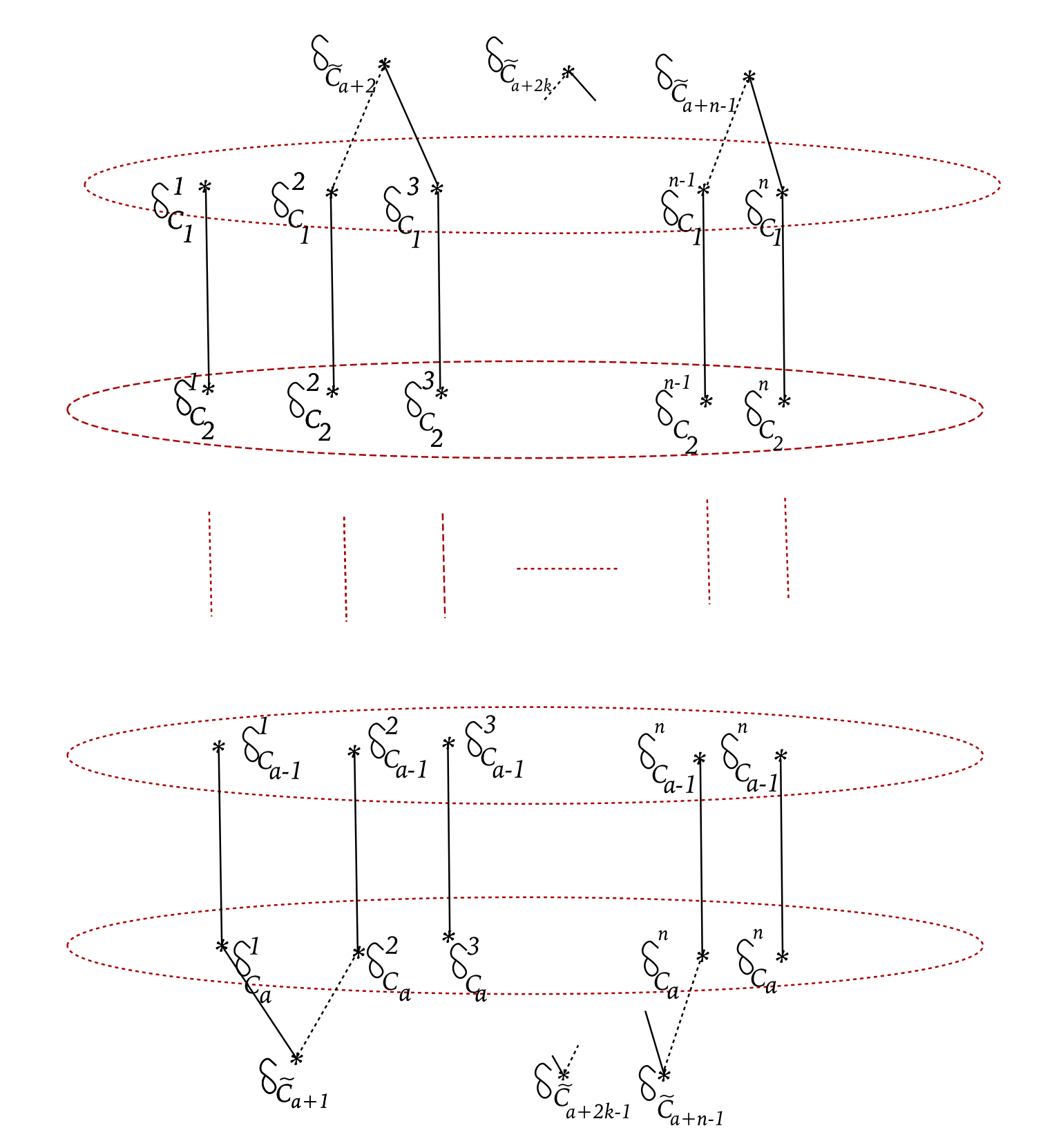}
\caption{Dynkin diagram of a pull-back polynomial when $n$ is odd.}\label{17;21:21}
\end{figure}

\begin{lem}\label{16:42}
The action of monodromy group $\pi_1(\mathbb{C}\setminus C\cup\tilde{C},b)$ on a tangency vanishing cycle generates all
$$
\delta_{\tilde{c}}\ , \ \tilde{c}\in\tilde{C}\ \ and\ \ \delta_{c}^i-\delta_c^j \ \ s.t\ \ c\in C \ , \ i,j=1,\dots ,n
$$ 	
where a tangency vanishing cycle is a vanishing cycle around a critical point of $R$.
\end{lem}
\begin{proof}
Each tangency vanishing cycle $\delta_{\tilde{c}_{a+2i-1}}$ (resp. $\delta_{\tilde{c}_{a+2i}}$ ) has an intersection with two vanishing cycles $\delta_{c_a}^{2i-1}$ and $\delta_{c_a}^{2i}$ (resp. $\delta_{c_1}^{2i}$ and $\delta_{c_1}^{2i+1}$ ) with different signs. By using Picard-Lefschetz formula, the action of monodromy $\lambda_{c_a}$ (resp. $\lambda_{c_1}$) on $\delta_{\tilde{c}_{a+2i-1}}$ (resp. $\delta_{\tilde{c}_{a+2i}}$ ) generates $\delta_{c_a}^{2i-1}-\delta_{c_a}^{2i}$ (resp. $\delta_{c_1}^{2i}-\delta_{c_1}^{2i+1}$ ). According to Lemma \ref{23:45}, the action of the monodromy group $$\pi:=<\lambda_c|\ c\in C>\subset\pi_1(\mathbb{C}\setminus((C\cup\tilde{C}),b ),$$ on $R_*(\delta_c^j)$ (for all $c\in C$ and $j$) generates zero homology group $H_0(g^{-1}(b),\mathbb{Z})$. Therefore, for a fixed $j$, the action of $\pi$ on $\delta_{c_a}^j$ or $\delta_{c_1}^j$  can generate all $\delta_c^j$. In other words, the action $\pi$ on $\delta_{c_a}^{2i}-\delta_{c_a}^{2i-1}$ (resp. $\delta_{c_1}^{2i+1}-\delta_{c_1}^{2i}$ ) generates all $\delta_c^{2i}-\delta_c^{2i-1}$ (resp. $\delta_c^{2i+1}-\delta_c^{2i}$). Since Dynkin diagram is connected, the action of monodromy $\lambda_{\tilde{c}_{a+2i}}$ ($\lambda_{\tilde{c}_{a+2i+1}}$) on $\delta_{c_{1}}^{2i}-\delta_{c}^{2i}$ can generate $\delta_{\tilde{c}_{a+2i}}$ (resp. $\delta_{\tilde{c}_{a+2i+1}}$). By repeating this procedure we can generate all $\delta_{\tilde{c}}$ because the degree of each vertex of the Dynkin diagram, is at most 2. Since the number of tangency vanishing cycles is $n-1$ we can generate independent cycles $\delta_c^{i+1}-\delta_{c}^i$ where $i=1,\dots,n-1$. Therefore these cycles can generate all $\delta_c^{j'}-\delta_c^j$. 
\end{proof}
\begin{coro}
The linear map $R_*:H_0((g\circ R)^{-1}(b),\mathbb{Z})\rightarrow H_0(g^{-1}(b),\mathbb{Z})$ is surjective and 
\[Ker(R_*)=<\pi_1(\mathbb{C}\setminus (C\cup\tilde{C}),b).\delta_{\tilde{c}}>,\]
where $<\pi_1(\mathbb{C}\setminus (C\cup\tilde{C}),b).\delta_{\tilde{c}}>$ is the group generated by the action of the monodromy group on the tangency vanishing cycle $\delta_{\tilde{c}}$.
\end{coro}
\begin{prop}\label{19:26}
The group generated by the action of the monodromy group $\pi_1(\mathbb{C}\setminus (C\cup\tilde{C}),b)$ on a vanishing cycle $\delta_c^j\in R_*^{-1}(\gamma_c)$, where $\gamma_c\in H_0(g^{-1}(b),\mathbb{Z})$,  is equal to  $H_0((g\circ R)^{-1}(b),\mathbb{Z})$.  
\end{prop}
\subsection{Direct Sum of Polynomials}
Let $F$ and $f$ be the functions as in
\eqref{12:35} 
and \eqref{12:51}. We are going to study the topology of a regular fiber of $f\circ F$.
\begin{nota}\label{11:38:17}
We denote by $C_1$ (resp. $C_2$) the set of critical values  of $g$ (resp. $h$), and also denote by $\tilde{C}_1$ (resp. $\tilde{C}_2$) the set of the images of the critical points of $R$ (resp. $S$) under $g\circ R$ (resp. $h\circ S$). Thus $C_1\cup\tilde{C}_1$ and $C_2\cup\tilde{C}_2$ are the set of critical values of $g\circ R$ and $h\circ S$ respectively. Without loss of generality, we can assume that $(C_1\cup\tilde{C}_1)\cap(C_2\cup\tilde{C}_2)=\emptyset$. 
\end{nota}
We take two systems of distinguished paths $\lambda_c$ relative to the functions $g\circ R$ and $h\circ S$, where $c\in (C_1\cup \tilde{C}_1)\cup (C_2\cup\tilde{C}_2)$ and $\lambda_c$ starts from $b=0$ and ends at $c$, see Figure \ref{11:04}. Note that for the function $h\circ S$ we choose a distinguished system of paths such that all of the paths are in the lower half plane, and they preserve the order of $C_2\cup\tilde{C}_2$ as in Notation  \ref{11:38:17}.
Let $\delta\in H_0((g\circ R)^{-1}(0),\mathbb{Z})$ and $\gamma\in H_0((h\circ S)^{-1}(0),\mathbb{Z})$ be two vanishing cycles along the paths $\lambda_{c}$ and $\lambda_{a}$ respectively. Let $t_s:[0,1]\rightarrow \mathbb{C}$ be a path defined by 
\[
t_s:=\begin{cases}
    \lambda_{c}(1-2s)&  0\leq s\leq \frac{1}{2}\ , \ c\in (C_1\cup\tilde{C}_1)\\
    \lambda_{a}(2s-1)  &  \frac{1}{2}\leq s\leq 1 \ , \ a\in (C_2\cup\tilde{C}_2).
    \end{cases}
\]
The cycle $\delta$ vanishes along $t_.^{-1}$ when $s$ tends to zero, and $\gamma$ vanishes along $t_.$ when s tends to 1.  
\begin{figure}[h]
  \centering
  \includegraphics[width=2 in]{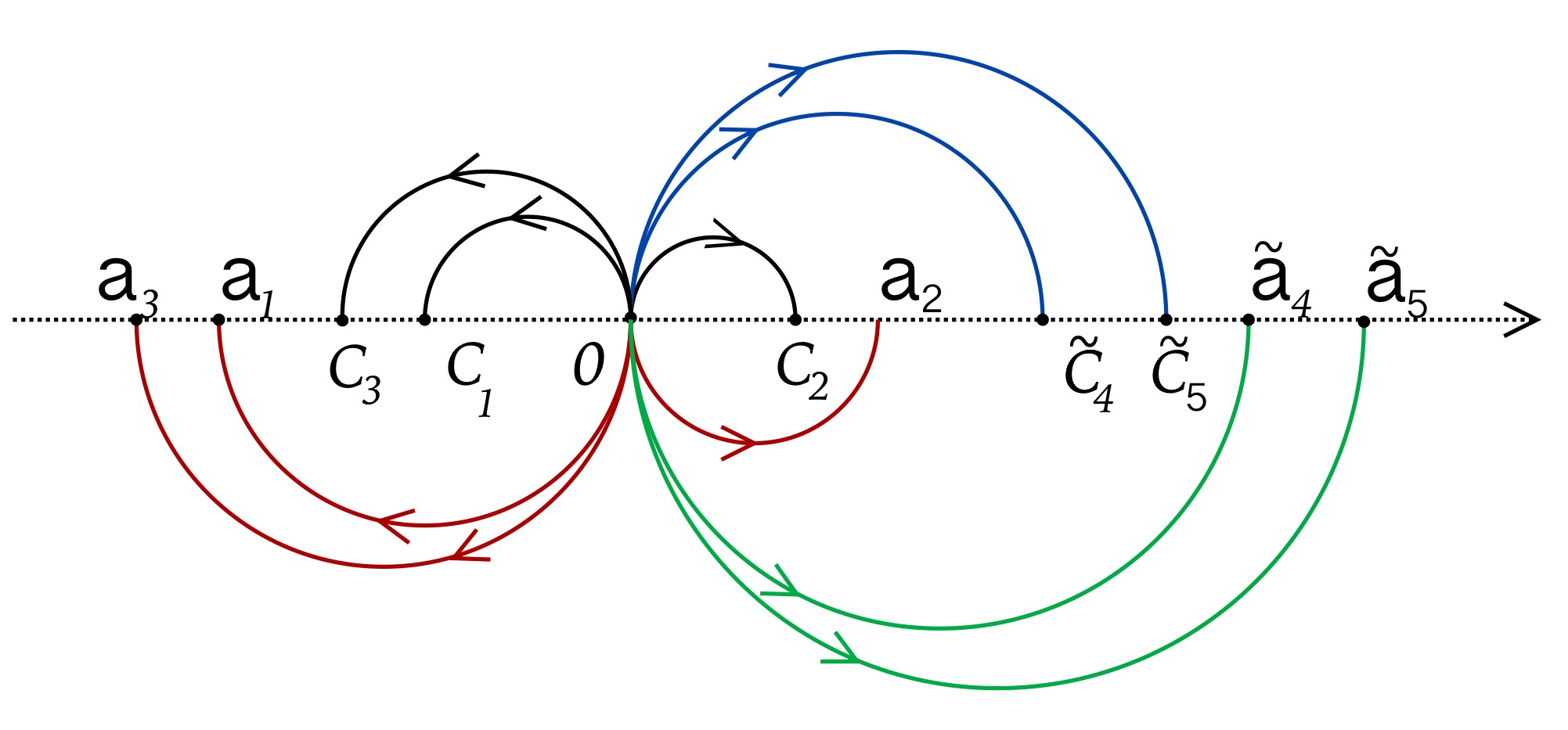}
  \caption{A distinguished system of paths where  $c_i\in C_1,a_i\in C_2$ and $\tilde{c}_i\in\tilde{C}_1,\tilde{a}_i\in\tilde{C}_2$}
 \label{11:04}
\end{figure}  
\begin{defi} The cycle
\[
\delta*\gamma\cong \delta*_{t_.}\gamma:=\cup_{s\in[0,1]} \delta_{t_s}\times\gamma _{t_s} \in H_1((f\circ F)^{-1}(0),\mathbb{Z})
\]
is called an oriented cycle. Note that its orientation changes when the direction of path $t_.$ is changed. The triple 
$(t_s,\delta,\gamma)=(t_s,\delta_{t_.},\gamma_{t_.})$ is called an admissible triple.
\end{defi}
\begin{figure}[h]
  \centering
  \includegraphics[width=3 in]{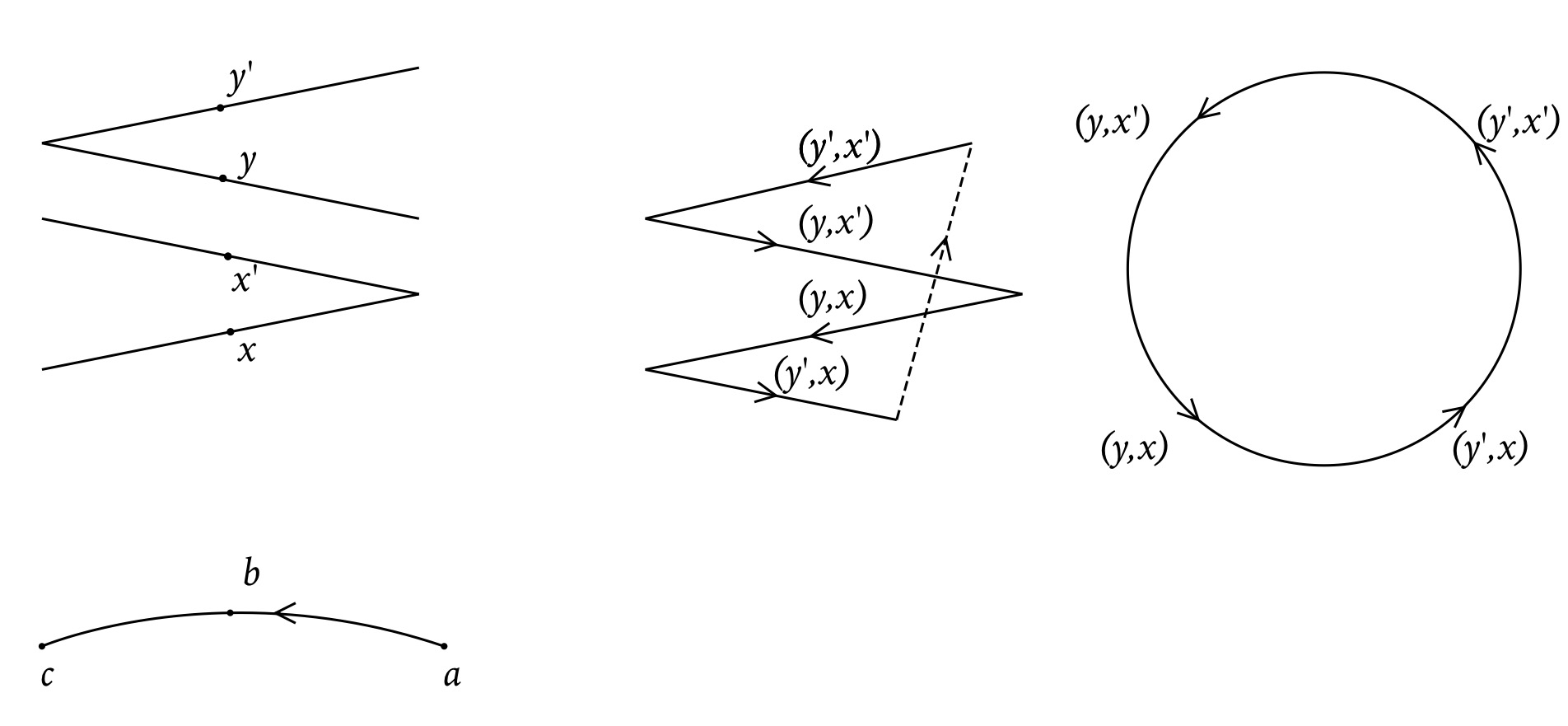}
  \caption{Joining of vanishing cycles }
 \label{11:04}
\end{figure}
Let
\[\delta_{c_i}^j\ ,\ \delta_{\tilde{c}_{a+k}}\in H_0((g\circ R)^{-1}(b),\mathbb{Z}) \ \ where  \ \ i=1,\dots,a \ , \ j=1,\dots,n\ ,\ k=1\dots,n-1,
\] 
and 
\[\gamma_{c_i}^j\ ,\ \gamma_{\tilde{c}_{a+k}}\in H_0((h\circ S)^{-1}(b),\mathbb{Z}) \ \ where  \ \ i=1,\dots,a \ , \ j=1,\dots,n\ ,\ k=1\dots,n-1,
\] 
be the corresponding distinguished basis of vanishing cycles.

\begin{theo}\label{15:39}
 The $\mathbb{Z}$-module $H_1((f\circ F)^{-1}(0),\mathbb{Z})$ is free and is generated by 
 $$
\alpha:=\delta*\gamma \ \ \ s.t \ \ \ \delta\in H_0((g\circ R)^{-1}(b),\mathbb{Z}) \ \ , \ \ 
\gamma\in H_0((h\circ S)^{-1}(b),\mathbb{Z}), 
 $$
 and where we have taken the admissible triples 
 $$
 (\lambda_c.\lambda_a^{-1}, \delta,\gamma) \ here \ c\in C_1\cup\tilde{C}_1 \ \ and \ \ a\in C_2\cup\tilde{C}_2. 
 $$
\end{theo}
See e.g. \cite[Ch 7]{15}.\\
Take $h\circ S= b'-(h'\circ S')$, where $b'$ is a fixed complex number and $h'\circ S'$ is a 
perturbation of $h\circ S$. The set of critical values of $h'\circ S'$ is denoted by $(C_2'\cup \tilde{C}_2')$ 
and therefore the set of critical values of $h\circ S$ is $C_2\cup\tilde{C}_2=b'-(C_2'\cup \tilde{C}_2')$. We define 
$(f\circ F)(x,y):=g\circ R(x) +h'\circ S'(y)$.
Assume that $(C_1\cup\tilde{C}_1)\cap(C_2'\cup\tilde{C}_2')=\emptyset$, and since the set of critical values of $f\circ F$ is $(C_1\cup\tilde{C}_1)+(C_2'\cup\tilde{C}_2')$ then  $b'$ is a regular value of $f\circ F$. Let 
$(t_s,\delta,\gamma)$ be an admissible triple where $t_s$ starts from $c$ and ends at $b'-a'$ (here $c\in C_1\cup\tilde{C}_1$ and $a'\in C_2'\cup\tilde{C}_2'$). Therefore the path $t_.+a'$ starts from $c+a'$ and ends at $b'$. For instance, see Figure \ref{18:45}: 
\begin{figure}[h]
  \centering
  \includegraphics[width=3.6 in]{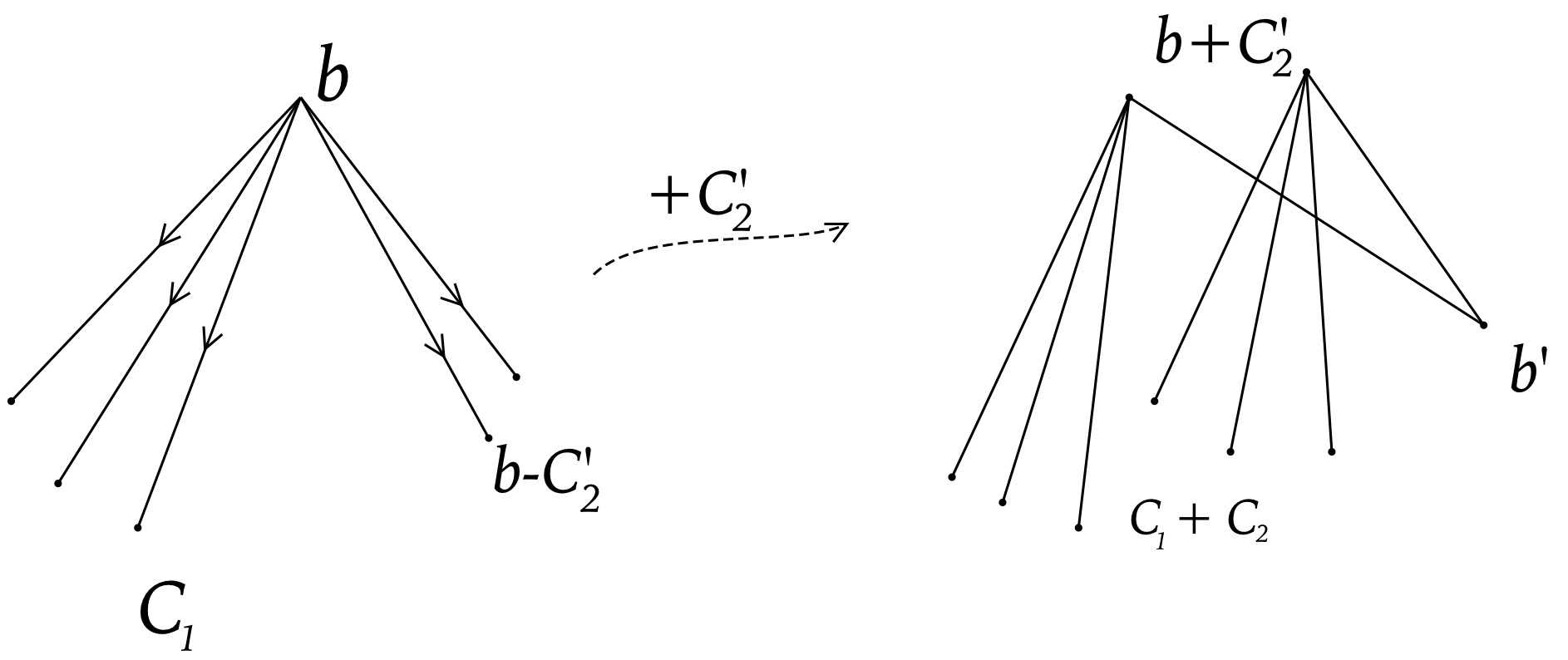}
  \caption{A distinguished system of paths }
 \label{18:45}
\end{figure}  
\begin{prop}
The topological cycle $\delta*\gamma$ is a vanishing cycle along the path $t_.+a'$ with respect to the fibration $f\circ F=t$.
\end{prop}
 See e.g. \cite[Ch 7]{15}.
 \begin{defi}\hfill
\begin{enumerate}\setlength\itemsep{.1em}
    \item A vanishing cycle around the critical point $p$ where $p\in F^{-1}(Sing(f))$ is called a pull-back vanishing cycle.  
    \item A vanishing cycle around a tangency critical point is called a tangency vanishing cycle.
     \item A vanishing cycle around a critical point $p$ where $p\in V(R_x)\cap V(S_x)$, is called an exceptional vanishing cycle.
\end{enumerate}
\end{defi}

For simplicity we denote by $\delta_i^j$ the cycle $\delta_{c_i}^j$ where $c_i\in C_1$, $i=1,\dots ,a$ and $j=1,\dots,n$ (resp. by $\gamma_i^j$ the cycle $\gamma_{a_i}^j$ where $a_i\in C_1$, $i=1,\dots ,a$ and $j=1,\dots,n$). Also, we denote by $\delta_k$ the cycle $\delta_{\tilde{c}_k}$ where $k=a+1,\dots,a+(n-1)$ (resp. by $\gamma_k$ the cycle $\gamma_{\tilde{c}_k}$ where $k=a+1,\dots,a+(n-1)$).

\begin{theo}\label{18:12}
Let $b=0$ be the regular value of the function $f$. Let $\delta_i $ where $i=1,\dots,a$ be the distinguished set of vanishing cycles in $H_0(g^{-1}(b),\mathbb{Z})$ also let $\gamma_j$ where $j=1,\dots,a$ be the distinguished set of vanishing cycles in $H_0(h^{-1}(b),\mathbb{Z})$. Therefore the intersection matrix of $H_1(f^{-1}(b),\mathbb{Z})$ in the basis \[
\delta_i*\gamma_j \  \ \  where \  \ \ \ i,j=1,2,\dots,a,
\]
is of the form 
 \[
<\delta_i*\gamma_j\ , \ \delta_l*\gamma_k>=
\]
\[
		\begin{cases}
			(-1)^{a} & \text{if}
			 \begin{cases} 
			 (i=l,k=j+1,j=odd)\vee\\
			 (j=k,l=i+1,i=even)\vee\\
			 (i=odd,j=even ,l=i+1\ or\ l=i-1,k=j+1)
			 \end{cases} \\
			(-1)^{a+1} & \text{if}
			 \begin{cases} 
			 (i=l,k=j+1,j=even)\vee\\
			 (j=k,l=i+1,i=odd)\vee\\
			 (i=even,j=odd,l=i-1\ or\ i+1,k=l+1)
			 \end{cases}\\
            0 & otherwise. 			
		\end{cases}
\]
\end{theo}See e.g. \cite[Ch 7]{15}, and  \cite{1}.\\
The Dynkin diagram of $f$ when $a$ is even, is shown in Figure \ref{17:29.11}. 
\begin{figure}[htp]
  \centering
  \includegraphics[width=3in]{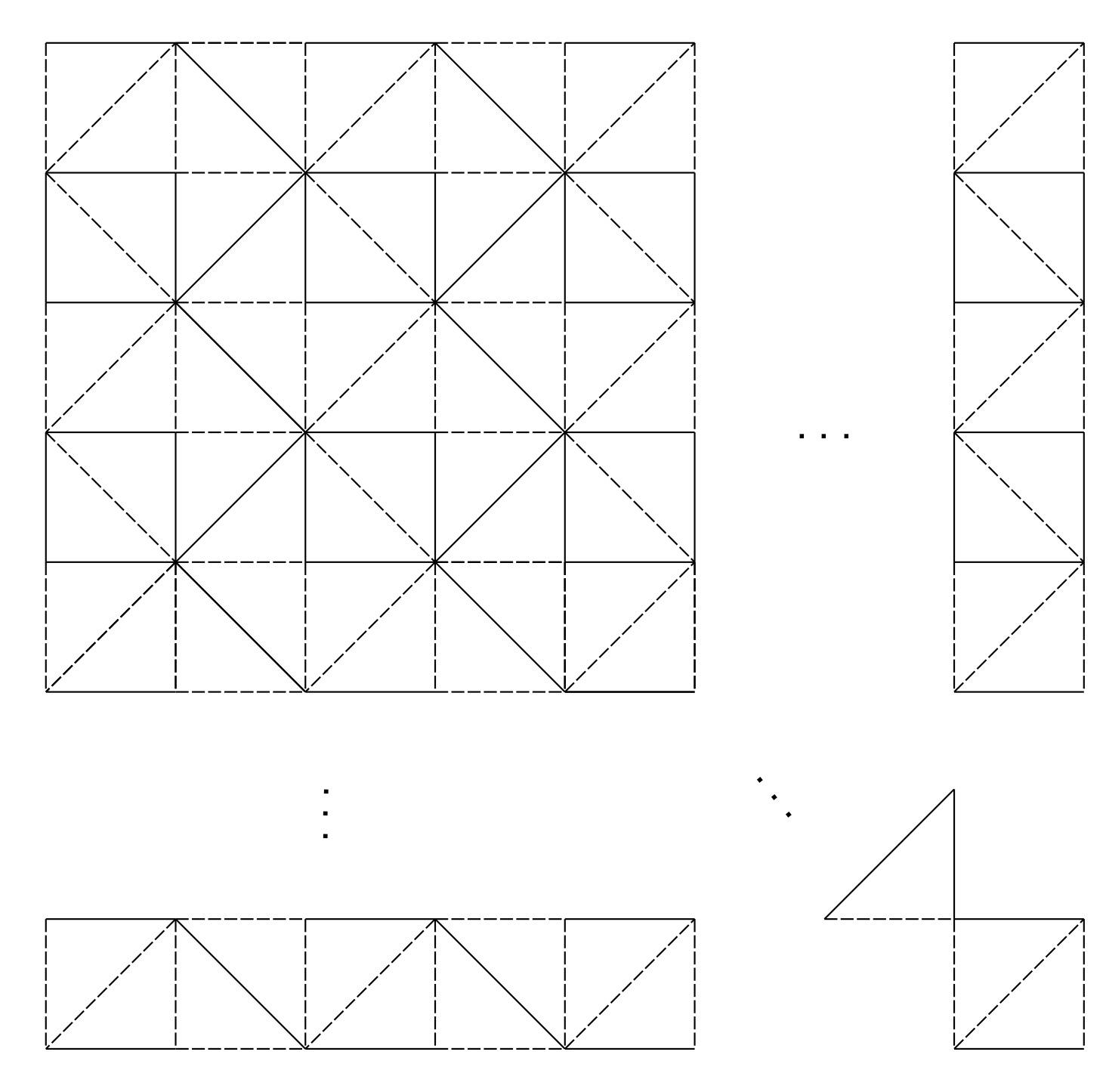}
  \caption{Dynkin diagram}
  \label{17:29.11}
 \end{figure} 
 
In Figure \ref{17:29.11}, according to the distinguished set of paths, the paths such as $t_{i,j}:=\lambda_j.\lambda_i^{-1}$  have transversal intersection with $t_{i,j+1}=\lambda_{j+1}.\lambda_{i}^{-1}$ (resp. $t_{i+1,j}:=\lambda_j.\lambda_{i+1}^{-1}$) at the point $b=0$, and $d(t_{ij})\wedge d(t_{i,j+1})=-d(t_{i,j+1})\wedge d(t_{i,j+2})$ (resp. $d(t_{ij})\wedge d(t_{i+1,j})=-d(t_{i+1,j})\wedge d(t_{i+2,j})$).

\begin{theo}\label{14:06}
For the regular value $b=0$ of $f\circ F=g(R)+h(S)$, we choose a distinguished set of vanishing cycles $\delta_i^j $ and $\ \delta_k$ where $i=1,2,\dots,a,j=1,\dots,n$ and $k=a+1,\dots,a+(n-1)$ (resp. $\gamma_i^j\ ,\ \gamma_k$ where $i=1,2,\dots,a,j=1,\dots,n$ and $k=a+1,\dots,a+(n-1)$) in $H_0((g\circ R)^{-1}(b),\mathbb{Z})$, (resp. $H_0((h\circ S)^{-1}(b),\mathbb{Z})$).  	
The intersection matrix in this basis 
$$
\delta_i^j*\gamma_{i'}^{j'} \ , \ \delta_k*\gamma_{i}^{j}\ , \ \delta_i^j*\gamma_k\ , \ \delta_k*\gamma_{k'} 
$$
for
$$
 i,i'=1,\dots,a \ , \ j,j'=1,\dots,n \ , \ k,k'=a+1,\dots,a+(n-1),
$$
of $H_1((f\circ F)^{-1}(0),\mathbb{Z})$ is given by
\[
<\delta_i^m*\gamma_j^s,\delta_l^{m'}*\gamma_k^{s'}>=<\delta_i^1*\gamma_j^1,\delta_l^1*\gamma_k^1>, \ \ \ for \ \ m=m'\ ,\ s=s' , 1\leq i,j,k,l\leq a,
\]
where \[<\delta_i^1*\gamma_j^1,\delta_l^1*\gamma_k^1>=(-1)^{n+1}<R(\delta_i^1)*S(\gamma_j^1),R(\delta_l^1)*S(\gamma_k^1)>.\]
Here, the intersection can be explained by using Theorem \ref{18:12} and
\[{
	\left\{
		\begin{array}{lll}
		<\delta_{a+i}^0*\gamma_j^1,\delta_{a+i}^0*\gamma_k^{s'}>=<\delta_a^s*\gamma_j^s,\delta_a^s*\gamma_k^{s'}>\\
		<\delta_{i}^m*\gamma_{a+j}^0,\delta_{l}^{m'}*\gamma_{a+j}^0>=<\delta_i^m*\gamma_a^m,\delta_l^{m'}*\gamma_a^{m}>,
		\end{array}
	\right.
}
\]
for the others we can use 
\[
<\delta_i^m*\gamma_a^s,\delta_i^m*\gamma_{a+s}^0>=<\delta_i^m*\gamma_{a+s}^0,\delta_{i}^{m}*\gamma_{a}^{s+1}>=
\]
\[
	\begin{cases}
		-1&\text{if }\begin{cases}  (n=2n'+1\ ,\ s=2t+1  )\vee \\(n=2n'\ ,\ a=2a'+1,\ s=2t+1)\end{cases}\\
	    1&\text{if } (n=2n'\ ,\ a=2a'+1\ ,\ s=2t+1 ),
		\end{cases}
\]
\[
<\delta_1^m*\gamma_j^s,\delta_{a+m}^0*\gamma_{j}^s>=<\delta_{a+m}^0*\gamma_{j}^s,\delta_1^{m+1}*\gamma_j^s>=
\]
\[
	\begin{cases}
		1&\text{if }\begin{cases}  (n=2n'+1\ ,\ m=2t+1  )\vee \\(n=2n'\ ,\ a=2a'+1,\ m=2t+1)\end{cases}\\
		-1&\text{if } (n=2n'\ ,\ a=2a'+1\ ,\ m=2t+1 ).
		\end{cases}
\]
\end{theo}
See e.g. \cite[Ch 7]{15}.\\
Since the dimension of $(f\circ F)^{-1}(b)$ is one, for the two vanishing cycles 
$\alpha,\beta\in H_1((f\circ F)^{-1}(b),\mathbb{Z})$, we have $<\alpha,\beta>=-<\beta,\alpha>$ and $<\alpha,\alpha>=0$, i.e. the intersection matrix is skew-symmetric.\\
The Dynkin diagram of $f\circ F$ when $n$ is odd and a=3 is shown in Figure \ref{Dyn}.
\begin{figure}[htp]
  \centering
  \includegraphics[width=3.5in]{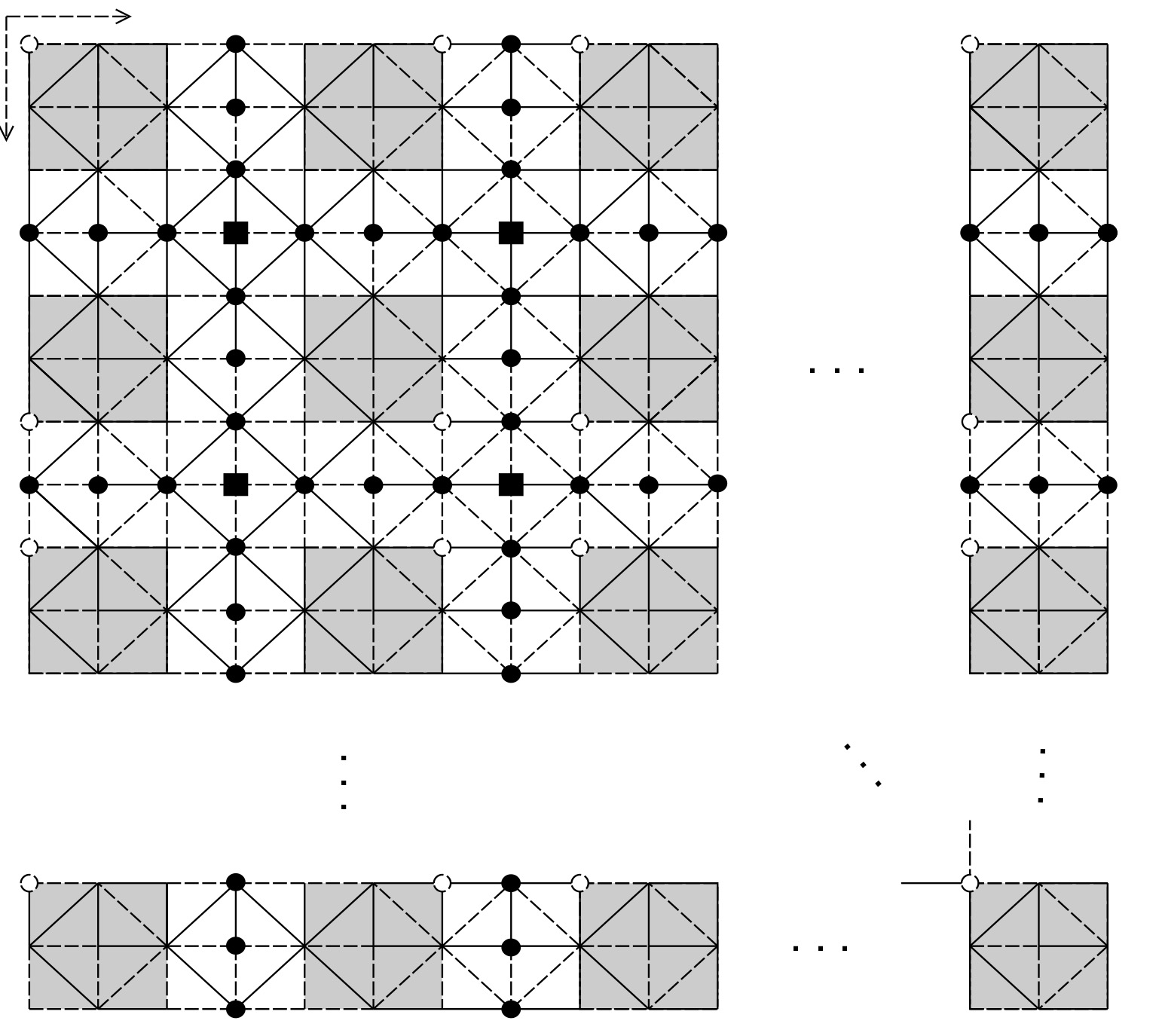}
  \caption{Dynkin diagram of $f\circ F$, when $n$ is odd and a=3 }
  \label{Dyn}
\end{figure} 
 
In Figure \ref{Dyn}, black vertices correspond to the tangency vanishing cycles, Squares vertices corresponding to the exceptional vanishing cycles and all the other vertices corresponding to pull-back vanishing cycles. The white cycle vertices correspond to some of the pull-back vanishing cycles with the same image under $F_*$. The direction of the intersections are to be considered from left to right and top to bottom in this figure.

\begin{defi}
An isomorphism of graphs $G$ and $H$ is a bijection between the vertex sets of $G$ and $H$
    $$f : V ( G )\rightarrow V ( H ), $$
such that any two vertices $u$ and $v$ are adjacent in $G$ if and only if $f(u)$ and $f(v)$ are adjacent in H.
\end{defi}
We denote by H (resp. G), the Dynkin diagram of $f\circ F$ (resp. $f$) with respect to the distinguished set of vanishing cycles related to the critical points of $f\circ F$ (resp. $f$). 
We consider the group generated by the action of the monodromy group $$\pi':=<\lambda_c|c\in C_1+C_2>\subset\pi_1(\mathbb{C}\setminus((C_1\cup C_2)+(\tilde{C}_1\cup\tilde{C}_2)),b),$$ on a pull-back vanishing cycle $\delta_c^i*\gamma_a^j$. We know that this group is generated by some pull-back vanishing cycles, hence it introduces a sub-graph of $H$ which is denoted by  $G_{ij}$.

For each $i,j=1,\dots,n$ the graph $G_{ij}$ is isomorphic to the graph $G$. Therefore if we remove the vertices corresponding to the tangency and exceptional vanishing cycles, then $H$ is divided into $n^2$ graphs $G_{ij}$.  
\begin{defi}
The cycle $\delta$ in a regular fiber $f^{-1}(b)$ is called simple if the homology group $H_1(f^{-1}(b),\mathbb{Z})$ is generated by the action of monodromy group $\pi_1(\mathbb{C} \setminus C,b)$ on $\delta$ (where $C$ is the set of critical values of $f$).
\end{defi}
\begin{theo}\label{14:37}
Each vanishing cycle (respective to the distinguished set of paths related to the critical values) in a regular fiber of $f$, is simple. 
\end{theo}
See e.g. \cite{16}.
\begin{theo}\label{asl}
 For the regular value $b=0$  of the function  $f\circ F=g\circ R+h\circ S$ (a composition of the functions which were defined in \eqref{12:35} and \eqref{12:51}), the action of the monodromy group on a tangency vanishing cycle generates
\[
\delta_{\tilde{c}}*\gamma_a \ \ , \ \ \delta_c*\gamma_{\tilde{a}} \ \ , \ \ \delta_{\tilde{c}}*\gamma_{\tilde{a}} \ \ and \ \  \delta_c^i*\gamma_a^j-\delta_c^{i'}*\gamma_a^{j'},
\]	
where 
$\tilde{c}\in\tilde{C}_1 \ , \ c\in C_1\ ,\ a\in C_2 , \ \tilde{a}\in\tilde{C}_2$ and $\ i,j,i',j'=1,2,\dots, n$.
\end{theo}
\begin{proof}\label{16:10} 
\begin{enumerate}\setlength\itemsep{.1em}
\item Suppose that $L$ is a line  of $D$. Each vanishing cycle around a critical point in $D$ is tangency or exceptional. 
Consider $f\circ F|_L-l$, when restriction of the functions $g\circ R$ or $h\circ S$ under $L$, is the constant value $l$. If $L=\{y= cons\}$ (resp. $L=\{x=cons\}$), then the cycles around the critical point in $L$ correspond to     
zero vanishing cycle $g\circ R$ (resp. $h\circ S$). Furthermore, a tangency vanishing cycle corresponds to a pull-back (zero) vanishing cycle. Thus by Proposition \ref{19:26}, the action of monodromy group on that cycle generates all the other vanishing cycles. This implies that, this action on our tangency vanishing cycle will generate the tangency and exceptional vanishing cycles around the critical points in $L$. The algebraic set $D$ consists of $n$ lines parallel with $x,y$ axes, so $ V(R_x)\cap V(S_y)\neq \emptyset$. Since the critical points in two vertical lines of $D$ have different values except the point in the intersection, the action of monodromy on a tangency vanishing cycle around a point of these lines can generate all the other vanishing cycles of these lines. In other words, this action on a tangency vanishing cycle can generate all the tangency and exceptional vanishing cycles.
\item If we remove the vertices of $H$ which is the graph corresponded to the tangency and exceptional vanishing cycles, then the new graph contains $n^2$ sub-graphs $G_{i,j}$ where $1\leq i,j\leq n$. The tangency vertex $\delta_t^i*\gamma_{a+j}^0$ (resp. $\delta_{a+i}^0*\gamma_t^j$) connects the sub-graphs $G_{i,j}$ to $G_{i,j+1}$(resp.$G_{i,j}$ to $G_{i+1,j}$). In general $\delta_t^i*\gamma_{a+j}^0$ connects 
$\delta_t^i*\gamma_a^j$ to $\delta_t^i*\gamma_a^{j+1}$ when $j$ is odd and also, it connects $\delta_t^i*\gamma_1^j$ to $\delta_t^i*\gamma_1^{j+1}$ when $j$ is even. According to Picard-Lefschetz formula the action of monodromy $\lambda_{t,a}$ (resp. $\lambda_{t,1}$) around the critical point $t+a$ (resp. $t+1$) when $j$ is odd (resp. $j$ is even) on the tangency vanishing cycle $\delta_t^i*\gamma_{a+j}^0$ is as follows:
\begin{eqnarray*}
\lambda_{t,a}(\delta_t^i*\gamma_{a+j}^0)&=&\delta_t^i*\gamma_{a+j}^0-\sum_{\alpha\in F_*^{-1}(\delta_t*\gamma_a)}<\alpha,\delta_t^i*\gamma_{a+j}^0>\alpha\\
&=& \delta_t^i*\gamma_{a+j}^0-(\delta_t^i*\gamma_{a}^j-\delta_t^i*\gamma_{a}^{j+1})
\end{eqnarray*} 
when $j$ is even, the action of monodromy $\lambda_{t,1}$ on $\delta_t^i*\gamma_{a+j}^0$  generates $\delta_t^i*\gamma_{1}^j-\delta_t^i*\gamma_{1}^{j+1}$.
According to Theorem \ref{14:37}, the action of the monodromy group $\pi$ on $\delta_t^i*\gamma_{a}^j$ generates vanishing cycles corresponding to the vertices in $G_{i,j}$. Therefore the action of the monodromy group $\pi'$ on the tangency vanishing cycle   
$\delta_t^i*\gamma_{a+j}^0$ generates the whole vanishing cycle $\delta_l^i*\gamma_{k}^j-\delta_l^i*\gamma_{k}^{j+1}$ where $1\leq l, k\leq a$. 
By the same process the action of the monodromy group $\pi$ on the tangency vanishing cycle   
$\delta_{a+i}^0*\gamma_t^j$ generates the whole vanishing cycle $\delta_l^i*\gamma_{k}^j-\delta_l^{i+1}*\gamma_{k}^{j}$ where $1\leq l, k\leq a $.
\end{enumerate}    
In general we can generate  \[\bigtriangledown:=\{\delta_c^{i+1}*\gamma_a^j-\delta_c^i*\gamma_a^j , \delta_c^i*\gamma_a^{j+1}-\delta_c^i*\gamma_a^j|i,j=1,\dots,n-1,c\in C_1,a\in C_2\}\] which is a basis for the group generated by 
\[
\delta_c^i*\gamma_a^j-\delta_c^{i'}*\gamma_a^{j'}\ \ \forall c\in C_1,\forall a\in C_2,\ \ i,j,i',j'=1,\dots,n.
\]  
\end{proof}

\begin{coro}\label{13:48;1}
The group which is generated by the action of the monodromy group on a tangency vanishing cycle can also be generated by 
\[
\delta_{\tilde{c}}*\gamma_a \ \ , \ \ \delta_c*\gamma_{\tilde{a}} \ \ , \ \ \delta_{\tilde{c}}*\gamma_{\tilde{a}} \ \ and \ \  \delta_c^i*\gamma_a^j-\delta_c^{i'}*\gamma_a^{j'}
\]	
where 
$\tilde{c}\in\tilde{C}_1 \ , \ c\in C_1\ ,\ a\in C_2 , \ \tilde{a}\in\tilde{C}_2$ and $\ i,j,i',j'=1,2,\dots, n$.
\end{coro}

\begin{theo}\label{18:00}
Let $F:\mathbb{C}^2\rightarrow\mathbb{C}^2$ be defined by $(x,y)\mapsto(R,S)$ as in \eqref{12:35} and let $f$ be a polynomial as in \eqref{12:51}. The linear map 
$$
F_*:H_1((f\circ F)^{-1}(b),\mathbb{Z})\rightarrow H_1((f)^{-1}(b),\mathbb{Z})
$$ is surjective, and $ker(F_*)$ is generated by 
the action of monodromy group $\pi_1(\mathbb{C}\setminus (C\cup\tilde{C}),b)$ on a tangency vanishing  cycle $\delta$.
\end{theo}
\begin{proof}
It is clear that for each $c\in C$ the cycles $\delta_c^i-\delta_c^j$ belongs to $ker(F_*)$ where $\delta_c^i,\delta_c^j$ are the pull-back vanishing cycles around the singularities with the value $c$. Each tangency vanishing cycle $\delta_t$ is divided into two paths with homotopic images under $F$ by $D$. Thus, $\delta_t\in ker(F_*)$. Each exceptional vanishing cycle $\delta_e$ is divided into 4 paths by $D$. The images of those 4 paths under $F$ are homotopic, hence $F_*(\delta_e)=0$. By using the Corollary \ref{13:48;1} we conclude that $<\pi_1(\mathbb{C}\setminus(C\cup\tilde{C}),b).\delta_t>\subset ker(F_*)$.
It is obvious that the morphism $F_*$ is surjective and so  
\begin{equation}\label{2:22}
\begin{split}
null(F_*)&=\#(V((g\circ R)_x)\cap V((h\circ S)_y))
-\#(V(g_x)\cap V(h_y))\\
&=(na+n-1)^2-a^2.
\end{split}
\end{equation}
If  $F_*^{-1}(\gamma_c)=\{\delta_c^i|i=1,\dots , n^2\}$, where $\gamma_c\in H_1(f^{-1}(b),\mathbb{Z})$, and since all the elements of the set $\triangle_c=:\{\delta_c^{i+1}-\delta_c^{i}|i=1,\dots ,n^2-1\}$ are independent, then its elements can generate all $\delta_c^i-\delta_c^j$ for all $i,j=1,\dots ,n^2$. The number of all the elements of the set $\cup_{c\in C}\triangle_c$ is $a^2(n^2-1)$. The tangency and exceptional vanishing cycles which are in $ker(F_*)$, are independent elements and their number is $2na(n-1)+(n-1)^2$. Therefore by considering the equality   
(\ref{2:22}) we have $<\pi_1(\mathbb{C}\setminus(C\cup\tilde{C}),b).\delta_t>= ker(F_*)$.
\end{proof}  

\textbf{Acknowledgment:} I would like to express my gratitude to my advisor professor Hossein Movasati whose contributions, inspirational suggestions and encouragement have helped me to coordinate my Ph.D. thesis especially in regards to writing this paper.


\smallskip
\leftline{Yadollah Zare}
\leftline{Instituto Nacional de Matem\'atica Pura e Aplicada-IMPA} 
\leftline{Rio de Janeiro}
\leftline{Email: yadollah2806@gmail.com, yadollah@impa.br} 
\end{document}